\documentclass{article}

\usepackage[english]{babel}
\usepackage[utf8]{inputenc}
\usepackage[letterpaper,top=2cm,bottom=2cm,left=3cm,right=3cm,marginparwidth=1.75cm]{geometry}
\usepackage{booktabs}

\usepackage{graphicx}
\usepackage{caption}
\usepackage{subcaption}

\usepackage{amsmath}
\usepackage{amsthm}
\usepackage{amsfonts}
\usepackage{amssymb}
\usepackage{mathtools} 
\usepackage{bm}
\usepackage{bbm} 
\usepackage{nicematrix}
\usepackage[normalem]{ulem}

\newtheorem{remark}{Remark}
\newtheorem{theorem}{Theorem}

\newtheorem{proposition}{Proposition}

\newcommand{\dd}{\mathrm{d}}

\newcommand{\R}{\mathbb{R}}

\usepackage{tikz}
\usetikzlibrary{decorations.markings}
\usetikzlibrary{arrows.meta} 

\usepackage{algorithm}
\usepackage{algpseudocode}

\usepackage{booktabs} 
\usepackage{multirow} 

\usepackage{todonotes}

\usepackage[colorlinks=true, allcolors=blue]{hyperref}

\usepackage{comment}

\definecolor{carrotorange} {rgb}{0.93, 0.57, 0.13}

\title{An Optimal Energy Production Problem with Energy Source Switching and Load Following Nuclear Power Plants}
\author{
  Fabio Baschetti \\
  \scriptsize{Department of Economics -- University of Verona} \\
  \scriptsize{Via Cantarane 24, 37129, Verona, Italy} \\
  \footnotesize{\texttt{fabio.baschetti@univr.it}} 
  \and
  Alessandro Gnoatto \\
  \scriptsize{Department of Economics -- University of Verona} \\
  \scriptsize{Via Cantarane 24, 37129, Verona, Italy} \\
  \footnotesize{\texttt{alessandro.gnoatto@univr.it}} 
  \and
  Athena Picarelli \\
  \scriptsize{Department of Economics -- University of Verona} \\
  \scriptsize{Via Cantarane 24, 37129, Verona, Italy} \\
  \footnotesize{\texttt{athena.picarelli@univr.it}}
}

\begin{document}
\maketitle
% \begin{abstract}
%     The paper deals with a optimal energy production problem where the producer aims to match at any time a prescribed demand managing two types of energy sources: a renewable energy source subject to randomness and seasonality on one hand, and a controllable deterministic system of production (e.g. nuclear) on the other hand. The problem is formulated as a optimal switching problem over three possible regimes representing the increasing/decreasing/constancy of the nuclear production rate.
% \end{abstract}

\begin{abstract}
The integration of weather-dependent renewable generation increases the volatility of residual demand and raises the value of dispatchable low-carbon flexibility. This paper studies the optimal operation of a load-following nuclear power plant owned by a producer that must balance stochastic residual demand while accounting for ramping limits and costly changes in operating regimes. Nuclear output can be increased, decreased, or kept constant, and the production decision is formulated as a finite-horizon optimal switching problem. We analyze both a closed-economy benchmark, where excess production cannot be sold and shortages require costly back-up generation, and an open-economy setting, where the producer can trade electricity at prices driven by aggregate market residual demand. The value functions are characterized as viscosity solutions of a system of Hamilton--Jacobi--Bellman quasi-variational inequalities, and optimal policies are computed using a monotone semi-Lagrangian scheme. The numerical results show how shortage costs, switching costs, ramping capability, and market access shape optimal nuclear load following. The analysis highlights the economic value of controllable low-carbon capacity in renewable-intensive systems and provides implications for flexibility remuneration, balancing-market design, and interconnection policy.
\end{abstract}

\noindent\textbf{Keywords:} Residual demand; Nuclear load following; Optimal switching; Renewable integration; Electricity markets; Balancing markets; Low-carbon flexibility.

\medskip

\noindent\textbf{JEL classification:} C61; C63; L94; Q41; Q42; Q48.

\section{Introduction}

The global commitment to decarbonization is reshaping electricity markets and placing increasing emphasis on the integration of renewable sources in the energy system. As wind and solar technologies expand, their weather-driven stochasticity creates operational and economic challenges for producers, system operators, and regulators. Electricity demand must be satisfied continuously, while renewable output fluctuates across intraday, weekly, and seasonal horizons. In this environment, the relevant balancing variable is residual demand, defined as electricity consumption net of renewable generation. Residual demand determines the need for dispatchable low-carbon generation, storage, imports, curtailment, or fossil-fuel back-up in the worst-case scenario.

Several technologies can contribute to managing renewable intermittency. Countries with large hydroelectric resources, such as Norway, can use reservoir and pumped-storage capacity as part of the balancing solution. Figure~\ref{fig:norway-hydro-modulation} illustrates this mechanism for Norway in 2019: reservoir hydropower generation varies substantially over the year and co-moves with residual demand, highlighting the role of stored water as a source of dispatchable flexibility.  Battery storage can provide fast and localized flexibility, but grid-scale storage remains costly when the objective is to cover prolonged periods of low renewable generation. These constraints create a role for low-carbon production technologies that are dispatchable, predictable, and available at scale. Existing nuclear fission technologies are one such source, while fusion energy may play a related role in the second half of the twenty-first century; see \cite{CABAL20171} and \cite{BUSTREO2019}. The role of nuclear fission in the green transition has recently been reconsidered in several developed economies. For example, the Joint Research Centre of the European Commission concludes that there is no ``science-based evidence that nuclear energy does more harm to human health or to the environment than other electricity production technologies already included in the Taxonomy as activities supporting climate change mitigation'' \cite[p.~3]{jrc2021}.

France provides a useful empirical motivation for the problem studied in this paper. It combines a historically large nuclear fleet, growing renewable production, and substantial cross-border market integration. Figure~\ref{fig:france-nuclear-modulation} reports French nuclear generation and residual demand for 2019 using RTE \'{e}CO2mix definitive data. The figure shows that nuclear production varies substantially over the year and co-moves with residual demand. In 2019, the mean daily nuclear range is about 3.3 GW, the 95th percentile of the daily range is about 9.2 GW, and the correlation between daily average nuclear generation and daily average residual demand is about 0.80. These patterns illustrate the economic relevance of nuclear modulation in a high-nuclear power system.

\begin{figure}[t]
    \centering
    \includegraphics[width=0.95\textwidth]{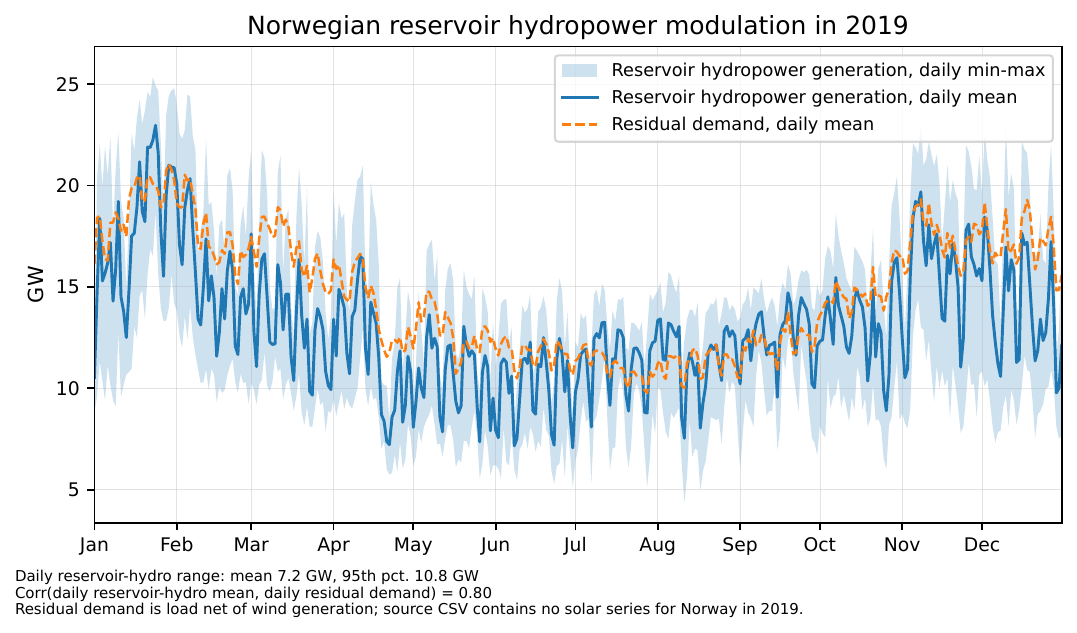}
    \caption{Norwegian reservoir hydropower generation and residual demand in 2019. The figure reports hourly Energy-Charts observations aggregated to daily series. The shaded region is the daily minimum--maximum range of reservoir hydropower generation, the solid line is daily mean reservoir hydropower generation, and the dashed line is daily mean residual demand, defined here as load net of wind generation. The observed co-movement illustrates how reservoir hydropower can provide flexible low-carbon generation in a hydro-rich system.}
\label{fig:norway-hydro-modulation}
\end{figure}

\begin{figure}[t]
    \centering
    \includegraphics[width=0.95\textwidth]{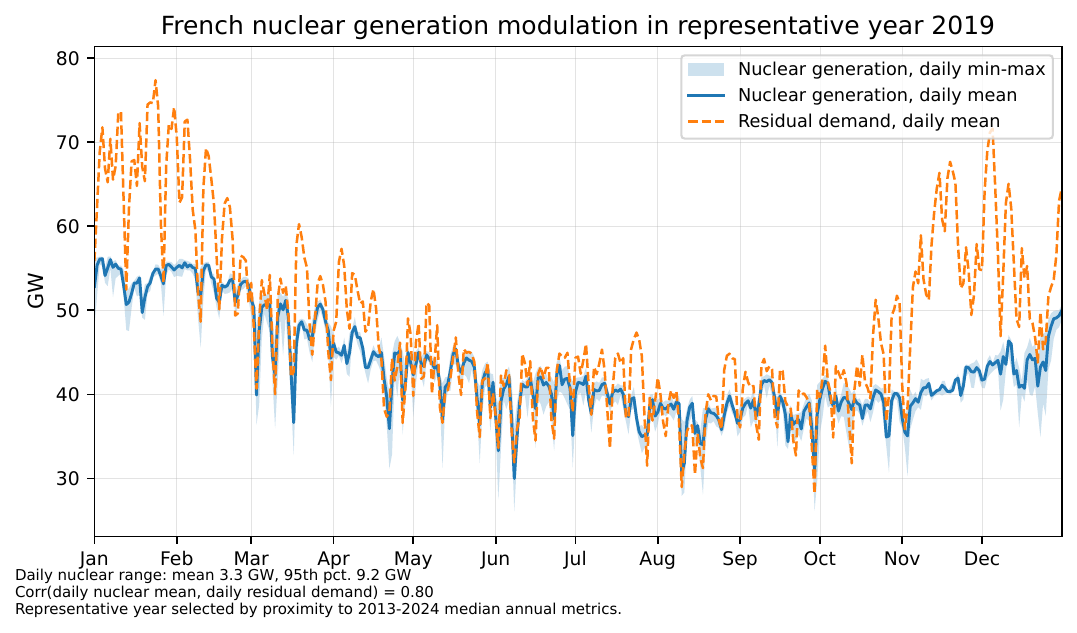}
    \caption{French nuclear generation and residual demand in a representative year. The figure reports daily observations constructed from RTE \'{e}CO2mix definitive half-hourly data. Nuclear generation is shown by its daily mean and daily minimum--maximum range. Residual demand is defined as electricity consumption minus wind and solar generation.}
    \label{fig:france-nuclear-modulation}
\end{figure}

In this work, we adopt the perspective of a large energy company operating a heterogeneous portfolio of production tools that includes renewable sources, nuclear power plants, and legacy fossil-fuel units. Renewable production is stochastic because it is driven by meteorological uncertainty. Electricity consumption is also stochastic around pronounced daily, weekly, and annual seasonal cycles. Nuclear and fossil generation are controllable from the point of view of the producer, although nuclear output is subject to technical limits, lower and upper production bounds, and finite maneuvering capability. Renewable output is dispatched with priority, reflecting both policy incentives and operational practice. The producer then uses nuclear generation to meet residual demand in the spirit of load following, while fossil generation is used only when nuclear capacity is insufficient or too costly to adjust.

The objective of the producer is to balance uncertain supply and demand while limiting reliance on carbon-intensive technologies and accounting for the cost of adjusting nuclear production. We model the maneuverability of the nuclear power plant and formulate an optimal switching problem governing nuclear production in real time. The plant can operate in three regimes: increasing production, decreasing production, or keeping production constant. Switching from one regime to another is costly, capturing the economic and technical frictions associated with maneuvering the plant. When nuclear production falls short of residual demand, the producer incurs a shortage cost that represents fossil back-up, imbalance penalties, emergency purchases, or reliability costs. When nuclear production exceeds residual demand, the producer incurs an excess-production cost in the closed-economy benchmark, where electricity cannot be sold externally or stored at scale.

We analyze two economic environments. The closed-economy case describes a producer that must balance domestic residual demand without access to an external electricity market. This benchmark isolates the physical and reliability value of nuclear flexibility. The open-economy case allows the producer to buy and sell electricity in an external market whose price depends on aggregate market residual demand. Market access changes the incentives for load following: excess nuclear production can have value when prices are high, while reducing production and purchasing electricity can be attractive when prices are low. The comparison between the two environments links nuclear operating decisions to market integration and to the design of flexibility incentives.

The paper makes four contributions. First, it provides a continuous-time stochastic formulation of nuclear load following under residual-demand uncertainty. Residual demand is modeled as a mean-reverting process with deterministic seasonal components, in line with electricity-market models in which load net of renewable generation is the key driver of conventional dispatch and prices. Second, it formulates the nuclear operating decision as a finite-horizon optimal switching problem. This representation is well suited to technologies whose output cannot be adjusted costlessly and instantaneously, because it explicitly accounts for ramping limits and costly changes in the direction of production. Third, it characterizes the value functions as the viscosity solution of a system of Hamilton--Jacobi--Bellman quasi-variational inequalities. Fourth, it develops a monotone semi-Lagrangian numerical scheme that exploits the switching structure of the problem and permits the computation of optimal policies under closed- and open-economy assumptions.

The numerical analysis studies how optimal nuclear operation depends on the economic and technical primitives of the model. In the closed economy, the plant follows residual demand more aggressively when shortage costs are high, maneuvering capability is large, or switching costs are low. Higher switching costs and lower shortage penalties generate wider inaction regions and smoother production profiles. In the open economy, optimal operation is additionally shaped by market-price regimes: the same local residual demand can lead to different nuclear decisions depending on whether external prices make purchases or sales attractive. These results are relevant for electricity systems that combine variable renewable generation with controllable low-carbon capacity, and for policy debates on flexibility remuneration, balancing markets, and interconnection.

The paper contributes to the literature in multiple directions. First, we study the interaction between nuclear generation and variable renewable energy. \cite{CANY2016135} analyze the compatibility between nuclear power and intermittent renewables in the French power mix. They show that higher wind and solar penetration reduces nuclear load factors and affects the levelized cost of nuclear generation, while also emphasizing that flexibility incentives may be needed if nuclear is to compete with gas-fired back-up. \cite{LOISEL2018} study nuclear load following in future European electricity systems and show that flexible nuclear operation can support renewable integration and improve welfare, although the benefits vary across systems depending on interconnection, technology mix, and the position of nuclear plants in the merit order. The technical and economic aspects of load following are also discussed by \cite{OECDNEA2011}, who emphasize that load following affects operating modes, fuel performance, component ageing, and plant economics. Scenario studies as in \cite{CABAL20171} and \cite{BUSTREO2019} finally examine how firm low-carbon technologies can complement renewables by reducing storage and back-up needs.

A second related literature studies residual demand and its effect on electricity prices. Residual demand---electricity demand net of renewable generation---is a key driver of conventional dispatch, imports, storage use, and wholesale prices. \cite{DO2021} provide an empirical analysis of residual electricity demand and show that residual demand is more stochastic and harder to forecast than total demand. This motivates modeling residual demand directly rather than separately modeling total demand and renewable generation. In electricity-market models, residual demand is also central to the merit-order mechanism because it determines which conventional technologies are needed to clear the market. 
 The increasing penetration of renewable generation strengthens the merit-order effect by displacing higher-cost conventional units and reducing wholesale electricity prices \cite{DeSianoSapio22}.
\cite{ACNHT:09} use residual demand in a risk-neutral model of electricity prices, while \cite{ANTWEILER2021} analyze the long-term merit-order effect of renewable generation on wholesale prices. More recently, \cite{ANTWEILER2025} study the viability of energy-only electricity markets in systems supplied by intermittent renewable generation and grid-scale storage.

A third strand of literature concerns stochastic optimization in power systems with renewable uncertainty. The rapid growth of weather-dependent generation has increased the need for models that account for uncertain supply, uncertain demand, balancing costs, and market prices. Reviews of stochastic optimization in renewable-energy applications emphasize that uncertainty representation and scenario generation are crucial for operational and planning decisions \cite{ZAKARIA2020, PVdams, ToufaniReview}. Related models consider stochastic dispatch, microgrid operation, demand response, storage, and market participation under uncertain renewable generation. Our approach is complementary to scenario-based stochastic programming because it formulates the problem in continuous time and uses dynamic programming to characterize a sequential, state-dependent operating policy affected by switching costs.

A final stream of related literature uses optimal switching to study systems that operate in a finite number of modes and where changing mode is costly. In energy applications, switching costs represent start-up costs, shutdown costs, adjustment costs, wear, or contractual frictions. The finite-horizon optimal switching problem is commonly characterized by a system of Hamilton--Jacobi--Bellman quasi-variational inequalities; see, for example, \cite{pham2009continuous}. Closely related work on stochastic control of electricity input under uncertain demand includes \cite{GKL19} and \cite{GKL23}, who study optimal control and chance-constrained formulations for supply systems with uncertain demand.

The existing literature establishes that nuclear flexibility can be technically feasible and economically valuable in systems with high renewable penetration, and that residual demand is a key variable for electricity-market outcomes. Less attention has been paid to the continuous-time operating problem of a nuclear producer that faces stochastic residual demand, costly regime changes, and different degrees of market access. The present paper addresses this gap by combining an economic model of residual-demand balancing with a stochastic optimal-switching formulation of nuclear load following. The framework links technology constraints to market incentives and produces policy-relevant comparative statics on shortage costs, curtailment costs, ramping capability, switching costs, and market integration.

The remainder of the paper is organized as follows. Section~\ref{sec:setting} introduces the model of nuclear production, residual demand, and operating costs. Section~\ref{sec:switching} formulates the optimal switching problem and derives the dynamic programming equation. Section~\ref{sec:numerics} presents the semi-Lagrangian numerical scheme. Section~\ref{sec:results} reports the numerical results for the closed- and open-economy cases. Section~\ref{sec:conclusion} concludes and discusses policy implications.

\section{Setting}\label{sec:setting}

We consider a regulated load-serving producer responsible for meeting domestic electricity demand over a finite horizon. Its generation portfolio comprises renewable technologies -- namely wind, solar photovoltaic, hydroelectric, and geothermal generation -- together with nuclear power plants (NPPs) and legacy fossil-fuel units. Since the model focuses on aggregate balancing rather than network constraints, production from each technology is represented at the system level and spatial heterogeneity is abstracted from.\\

Renewable generation has priority in the merit order and is dispatched first. Nuclear generation is used to meet residual demand whenever possible, while fossil-fuel generation represents a costly back-up option. The producer's short-run problem is therefore to choose the operating regime of the nuclear plant so as to balance stochastic residual demand while accounting for production limits, ramping constraints, and switching costs.\\

The ultimate objective of the energy company is to match the residual demand (net of renewables) through nuclear production in a way that is optimal according to some objective to be specified below. \\

The company is interpreted as a regulated load-serving entity. Its retail revenues are either fixed or predetermined by long-term contracts and therefore do not enter the short-run control problem. However, failing to meet residual demand generates a direct economic cost, because the company must either activate expensive backup generation, buy emergency electricity, or incur penalties associated with involuntary load shedding. A shortage penalty in the objective stems as a reduced-form representation of reliability costs. \\

We formalize the activity of the firm as the solution of an optimal switching problem over a finite horizon. We first assume that the economy is closed and the company has no interaction with the market. This means that the firm cannot purchase electricity if it fails to meet the internal demand or sell excess production.  Conversely, if the economy is open, then the company can trade electricity internationally. In this sense, it is either a price taker whose actions do not influence the market price of electricity, or it participates actively to the definition of a global equilibrium by means of its nuclear output. In practice, it is reasonable to assume that the role of the company depends on its size. \\

We begin by introducing the basic building blocks in our framework -- the nuclear output and the residual demand -- in terms of their instantaneous dynamics. For this, we fix a filtered probability space $(\Omega, \mathcal{F}, \mathbb{F} = (\mathcal{F}_t)_{t\in [0,T]},\mathbb{P})$.

\subsection{Nuclear production}

The agent controls the nuclear output. 
Production within the nuclear power plant depends on physical and regulatory considerations, but is otherwise deterministic. 
We assume the following dynamics specification:
\begin{equation}\label{eq:NP}
    \dd P_t = \mu(P_t,I_t) \dd t, \qquad \mu(p,i) =  \max(i,0) \mathbbm{1}_{ \{ p < P^{\max} \}} + \min(i,0)  \mathbbm{1}_{ \{ p > P^{\min} \} } %i \mathbbm{1}_{p \in [P_{\min},P_{\max}]} 
\end{equation}
where $i \in \mathbbm{I} = \{0, +r, -r\}$ identifies different production regimes. In particular,
\begin{itemize}
    \item the case $i=0$ corresponds to production at a constant rate;
    \item the case $i=+r$ corresponds to an increase in production at rate $r$, up to eventually reaching $P_{\max}$;
    \item the case $i=-r$ corresponds to a decrease in production at rate $r$, up to eventually reaching $P_{\min}$.
\end{itemize}
Such a rate $r>0$ defines the maneuverability of the NPP and is imposed on the producer by the relevant international agency. In practice, what the authority sets is actually the maximum allowable rate $\bar{r}$ so that the firm can choose any $r \in (0,\bar{r}]$ depending on the circumstances and adjust its value over time. We do not account for this flexibility and instead treat $r$ as a fixed constant for the sake of simplicity. \\

The drift $\mu$ in \eqref{eq:NP} ensures that $P_t \in [P^{\text{min}}, P^{\text{max}}]$ for every $t\geq 0$. In fact, it is economically inconvenient for the firm to shut down the NPP completely (or operate it at very low capacity), and an upper bound is set for security reasons to not be exceeded under standard operating programs.
\\

A production strategy within the NPP is determined by a double sequence $\alpha = (\tau_n, \iota_n)_{n\in \mathbb N}$, where
$(\tau_n)$ is an increasing sequence of stopping times representing the decision on "when to switch", and the $\iota_n$s are $\mathcal{F}_{\tau_n}$-measurable random variables taking values in $\mathbb{I}$, thus answering the question on "where to switch". In other terms, $\tau_n\in [0,T]$ is the (random) time when the decision to switch from regime $\iota_{n-1}$ to regime $\iota_n$ occurs. \\

We denote by $\mathcal{A}$ the set of switching controls. Then, given an initial regime $i \in \mathbb{I}$ and a control $\alpha = (\tau_n, \iota_n)_{n \in \mathbb{N}} \in \mathcal{A}$, the càd-làg process 
\begin{equation*}
I^i_t = \sum_{n\in \mathbb N} \iota_n {\mathbbm 1}_{[\tau_n, \tau_{n+1})}(t), \qquad t \geq 0, \quad I^i_{0^-} = i    
\end{equation*}
identifies the current regime at time $t$. Therefore, it is clear that $I^i_0 = \iota_1 \neq i$ is also a possibility as associated with a jump at time $\tau_0 = 0$.

\subsection{Residual demand}

In analogy with \cite{ACNHT:09}, we directly model the residual demand net of renewable energy sources: $Y_t = D_t - R_t$. For this we write
\begin{equation} \label{eqn:res_dem_OU}
\begin{aligned}
    & \dd Y_t = \kappa(\theta_t-Y_t)\dd t + \nu \dd W^Y_t \\
    & \theta_t = \beta + \sum_j (\zeta_j \cos(\omega_j t) + \eta_j \sin(\omega_j t) ) 
\end{aligned}    
\end{equation}
where $\beta$ is a baseline demand and the $\omega_j$s represent the fundamental frequencies that capture the daily, weekly, and yearly cycles. \\

The Ornstein--Uhlenbeck specification provides a parsimonious reduced-form model of residual demand. It captures mean reversion around deterministic seasonal components while allowing residual demand to take negative values, which may occur when renewable production exceeds total electricity consumption. The purpose of this specification is to obtain a tractable stochastic state variable for the optimal switching problem. We therefore restrict attention to a continuous diffusion framework, which is the setting used for the dynamic programming equation and the semi-Lagrangian numerical scheme below. See Appendix \ref{sec:estimation} for further discussions on the model.

\subsection{The optimal switching problem}\label{sec:switching}

We now formulate our optimal switching problem both in the closed economy and in an open economy.

\subsubsection{Closed economy}

A closed economy is one where the firm can only rely on its production tools in the attempt to satisfy the residual demand. Therefore, the only way to bridge a production deficit is to resort to fossil fuels. We strongly penalize this last circumstance in view of a green transition. On the other hand, there is no direct incentive for the company to operate the nuclear power plant at full capacity if not strictly required by the domestic customers. \\

In the course of its operations, the firm incurs the following social, environmental, and economic costs:
\begin{itemize}
    \item Excess production $\lambda_1 (P_t - Y_t)^+$. Because electricity cannot be stored, over-production is a cost.
    \item Production deficit $\lambda_2 (Y_t - P_t)^+$. Shortfalls in the nuclear output necessitate resorting to fossil fuels in order to avoid blackouts. 
    % The choice $\lambda_2 >> \lambda_1$ aims at avoiding that.
    The coefficient $\lambda_2$ should not be interpreted as an arbitrary preference parameter. It represents the marginal economic cost of a shortage. Depending on the institutional setting, it may include the marginal cost of fossil backup generation, imbalance penalties, lost retail margins, regulatory compensation payments, and, in the extreme case of involuntary curtailment, the value of lost load. Therefore the condition $\lambda_2 \gg \lambda_1$ reflects the empirical fact that shortage and reliability costs are typically much larger than curtailment or overproduction costs.
    \item Operating costs $\gamma_1 P_t$ include all the expenses the firm has to bear to keep the NPP running, ranging from component wear and scheduled maintenance to ensuring the continuity of everyday activities. It is therefore reasonable to model such costs as proportional to the production rate.
\item Cost of switching \(\sum_n c_{\iota_{n-1},\iota_n}\). 
Switching from regime \(i\) to regime \(j\in\mathbb I\) entails an instantaneous cost \(c_{i,j}\). We assume that
\begin{align*}
    & c_{i,j}\geq 0 \quad \text{for all} \ i \neq j \\
    & c_{i,i} = 0
\end{align*}
and the matrix % that the switching-cost matrix
\[
C =
\begin{pmatrix}
0 & c_{-r,0} & c_{-r,+r} \\
c_{0,-r} & 0 & c_{0,+r} \\
c_{+r,-r} & c_{+r,0} & 0
\end{pmatrix}%,
%\qquad
%\mathbb I=\{-r,0,+r\},
\]
satisfies the triangle inequality: % condition 
$$c_{i,k} < c_{i,j} + c_{j,k} \quad \text{ for } j \neq i,k.$$ 
In particular, this implies a no-free-loop condition
% Notice that this  implies the  no-free-loop condition:
\begin{equation}
\label{eq:noFreeLoop}
c_{i_0,i_1}+c_{i_1,i_2}+\cdots+c_{i_{n-1},i_0}>0
\end{equation}
% for every non-trivial cycle \(i_0,i_1,\ldots,i_{n-1},i_0\) in \(\mathbb I\) 
which rules out cost-free instantaneous cycles of switches and ensures that the switching problem is well posed. \\ Finally, switching costs are also assumed to be constant and independent of time and state.
\end{itemize}

We therefore set $X = (P, Y) \in Q:=[P_{\min},P_{\max}]\times \R$. Then, given an initial state-regime $(x,i) \in Q \times \mathbb{I}$ and a switching control $\alpha \in \mathcal{A}$, the controlled process $X^{x,i}$ is the solution to the stochastic differential equation below
\begin{equation}\label{eqn:X_dyn}
    dX_t = b(t,X_t,I^i_t) \dd t + \sigma \dd W_t,
\end{equation}
where 
\begin{equation*}
    b(t,(p,y),i) = 
    \begin{bmatrix}
         \max(i,0) \mathbbm{1}_{ \{ p < P^{\max} \}} + \min(i,0)  \mathbbm{1}_{ \{ p > P^{\min} \} } \\
        \kappa(\theta_t-y)
    \end{bmatrix},
    \qquad 
    \sigma =
    \begin{bmatrix}
        0 & 0 \\
        0 & \nu
    \end{bmatrix}
\end{equation*}
and  $W = [W^P,W^Y]$ is a two-dimensional Brownian motion. The covariance structure of $W$ is consistent with the fact that nuclear production is deterministic. \\

%Regimes are characterized by running costs $f : [0,T] \times \mathbb{R}^d \times \mathbb{I} \to \mathbb{R}$ and terminal costs $g : \mathbb{R}^d \times \mathbb{I} \to \mathbb{R}$. One usually writes $f^i_t = f(\cdot,\cdot,i)$ and $g^i = g(\cdot,i)$ to indicate the costs associated with the particular regime $i \in \mathbb{I}$. However, no explicit dependence on the regime - nor time (for $f$) - is in place in our case:

Running costs $f:Q\to\R$ are given by
\begin{align}
    & f(x) \equiv f((p,y)) = \lambda_1 (p - y)^+ + \lambda_2 (y - p)^+ + \gamma_1 p\label{eqn:running_cost}
\end{align} 
and terminal costs are such that $g(x) \equiv 0$.

% The expected total cost for the energy company when running the system described above starting at time $t \in [0,T]$ from point $(x,i)\in Q\times \mathbb I$, and enacting the control strategy $\alpha = (\tau_n, \iota_n)_{n \in \mathbb{N}}$, is stated in the following:
% \begin{equation*}
%     J^\alpha(t,x,i) = \mathbb{E} \left[ \int_t^T f(X_s^{x,i}) ds + \sum_{\tau_n \in [t,T)} c_{\iota_{n-1},\iota_n} \right].
% \end{equation*}

\subsubsection{Open economy}\label{sec:open}

In the previous setting, the unique drivers for agent decisions come from the need to satisfy the demand and avoid unnecessary overproduction. \\

We now allow the producer to trade electricity in an external market. Its ultimate objective remains to fulfill its "local" demand $\{Y_{t}\}_{t\geq 0}$, but it can now buy electricity internationally. Also, overproduction can now be sold. \\ 

Let $\{S^{\text{s}}_t\}_{t\geq 0}$ (resp. $\{S^{\text{b}}_t\}_{t\geq 0}$) be the process describing the market price for selling (resp. buying)  electricity. 
The company will incur in:
\begin{itemize}
    \item Gains from excess production $ S^{\text{s}}_t (P_t - Y_{t})^+$. Over-production is no longer wasted, but can be sold at price $S^{\text{s}}_t$.
    \item Costs of a production deficit $S^{\text{b}}_t (Y_{t} - P_t)^+$. The marginal demand after nuclear production can be covered by the excess production available from foreign operators at price $S^{\text{b}}_t$.
\end{itemize}
As usual, the company still faces operating costs $\gamma_1 P_t$ and switching costs $\sum_n c_{\iota_{n-1},\iota_n}$. \\

We consider three price scenarios, depending on the residual demand on the market $\{M_{t}\}_{t\geq 0}$.
Let $s_L, s_M, s_H \in \R$ be the price levels associated with different clearing resources. In particular, we assume that $0\leq s_L< s_M <s_H$ so as to reflect the marginal cost of renewable, nuclear, and fossil energy, respectively. We also denote $a>0$ the total nuclear capacity in the system. Once a constant bid-ask spread $\delta >0$ has been fixed, we have that
$$
S^{\text{s}}_t = \psi(M_{t}) := \begin{cases}
   s_{L} & \text{if}\quad M_{t}\leq 0\\
    s_{M} & \text{if}\quad M_{t}\in  (0, a] \\
     s_{H} & \text{if}\quad M_{t}> a,
\end{cases}
\qquad \text{and}\qquad S^{\text{b}}_t = S^{\text{s}}_t +\delta.
$$

We therefore set $X = (P, Y, M) \in Q:=[P_{\min},P_{\max}]\times \R\times \R$. Then, given an initial state-regime $(x,i) \in Q \times \mathbb{I}$ and a switching control $\alpha \in \mathcal{A}$, the controlled process $X^{x,i}$ is the solution to the stochastic differential equation below
\begin{equation}\label{eqn:X_dyn_open}
    dX_t = b(t,X_t,I^i_t) \dd t + \sigma \dd W_t
\end{equation}
where 
\begin{equation*}
    b(t,(p,y,m),i) = 
    \begin{bmatrix}
         \max(i,0) \mathbbm{1}_{ \{ p < P^{\max} \}} + \min(i,0)  \mathbbm{1}_{ \{ p > P^{\min} \}} \\
        \kappa_Y(\theta_{Y,t}-y)\\
        \kappa_M(\theta_{M,t}-m)
    \end{bmatrix},
    \qquad 
    \sigma =
    \begin{bmatrix}
        0 & 0 & 0\\
        0 & \nu_Y & 0\\
        0 & 0 & \nu_M
    \end{bmatrix}
\end{equation*}
and $W = [W^P, W^{Y}, W^{M}]$ is a three-dimensional Brownian motion. \\

\textbf{Price taker} \\

In the price-taker case, the cost of buying and selling electricity depends only on the market residual demand and not on the producer's own nuclear output. \\

The running cost $f:Q\to\R$ is given by
\begin{equation} \label{runcosts_pricetaker}
    f(x) \equiv f((p,y,m)) = (\psi(m) +\delta) (y - p)^+ - \psi(m) (p - y)^+  + \gamma_1 p  
\end{equation} 

\textbf{Price maker} \\

Our simple framework also accommodates the case where the firm concurs in defining an equilibrium price. For this, we write 
\begin{equation*}
    q_t := Y_t - P_t 
\end{equation*}
for the (im)balance of the firm, and define
\begin{equation*}
    Z_t : = M_t + q_t.
\end{equation*}
Then
\begin{equation*}
    S^s_t = \psi(Z_t) \qquad \text{and} \qquad S^b_t = S^s_t + \delta.
\end{equation*}
This gives
\begin{equation} \label{runcosts_pricemaker}
    f(x) \equiv f\left( (p, y, m) \right)=\left(\psi\left(m+y-p\right)+\delta\right)\left(y-p\right)^{+}-\psi\left(m+y-p\right)\left(p-y\right)^{+}+\gamma_1 p
\end{equation}
thus clarifying that the nuclear production of the firm now modifies the market price via a feedback on $Z_t$. \\

In both cases, we still have $g \equiv 0$. \\ 

\subsection{The Dynamic Programming equation}

 In both scenarios above, the expected total cost for the energy company when running the system starting at time $t \in [0,T]$ from point $(x,i)\in Q\times \mathbb I$, and enacting the control strategy $\alpha = (\tau_n, \iota_n)_{n \in \mathbb{N}}$, is stated in the following:
\begin{equation*}
    J^\alpha(t,x,i) = \mathbb{E} \left[ \int_t^T f(X_s^{x,i}) \dd s + \sum_{\tau_n \in [t,T)} c_{\iota_{n-1},\iota_n} \right].
\end{equation*}

Minimizing such an expectation over $\mathcal{A}$ defines the value functions  $v^i:[0,T] \times Q \to \R$, $i\in\mathbb I$,
\begin{equation}\label{eqn:value_functions}
    v^i(t,x) = \inf_{\alpha \in \mathcal{A}} J^\alpha(t,x,i), \qquad i \in \mathbb{I},
\end{equation}
for which we provide a PDE characterization via the dynamic programming principle (DPP), that now reads as
\begin{equation*}
    v^i(t,x) = \inf_{\alpha \in \mathcal{A}} \mathbb{E} \left[ \int_t^\vartheta f(X^{x,i}_s) \dd s + \sum_{\tau_n \in [t,\vartheta)} c_{\iota_{n-1},\iota_n} + v^{I^i_\vartheta}(\vartheta,X^{x,i}_\vartheta) \right], 
\end{equation*}
where $\vartheta \in (t,T)$ is any $\mathbb F$-stopping time. \\ 
By means of the DPP  value functions $(v_i)_{i\in \mathbb  I}$ in \eqref{eqn:value_functions} can be characterised  by the following system of $|\mathbb I|$ Hamilton-Jacobi-Bellman quasi variational inequalities (HJB-QVIs):
\begin{equation*}
\begin{aligned}
& \max \left\{ -\left( \partial_t v^i + \mathcal{L}^iv^i \right) (t,x) - f(x), v^i(t,x) - \min_{j\neq i} (v^j(t,x) + c_{i,j}) \right\} = 0, \qquad (t,x) \in [0,T)\times Q, i \in \mathbb{I}, 
%$\\
%& \text{s.t. } v^i(T,x) = 0
\end{aligned}
\end{equation*}
where
\begin{equation*}
    \mathcal{L}^i \varphi = b(\cdot,i) \nabla_x \varphi + \frac{1}{2} \text{Tr}[\sigma\sigma^\top\nabla_x^2 \varphi]
\end{equation*}
is the generator of $X$ in the regime $i$, coupled with the terminal conditions 
$$
v^i(T,x) = 0 \qquad x\in Q, i\in\mathbb{I}.
$$ 
The drift in \eqref{eq:NP} is discontinuous at the production bounds because of the indicator functions. This discontinuity is convenient for modelling the hard production constraints \(P^{\min}\) and \(P^{\max}\), but it falls outside the standard Lipschitz framework used in the Markovian theory of optimal switching. For the theoretical characterization, we therefore work with a smooth approximation of the production drift. For \(K>0\), we define
\[
\chi_K^{+}(p)
=
\frac{1}{1+\exp\{K(p-P^{\max})\}},
\qquad
\chi_K^{-}(p)
:=
\frac{1}{1+\exp\{-K(p-P^{\min})\}}
\]
and replace \(\mu\) by
\[
\mu_K(p,i)
:=
\max(i,0)\chi_K^{+}(p)
+
\min(i,0)\chi_K^{-}(p).
\]
For every fixed \(K\), the map \(\mu_K(\cdot,i)\) is bounded and Lipschitz continuous for each \(i\in\mathbb I\). Moreover, \(\mu_K\) converges pointwise to the original drift away from the production bounds as \(K\to\infty\). Hence the regularized state dynamics obtained replacing $\mu$ with $\mu_K$ satisfy the standard regularity assumptions needed for existence and uniqueness of strong solution to equation \eqref{eq:NP}. \\
Moreover, in the open-economy case, the price function \(\psi\) is discontinuous at the thresholds
\(0\) and \(a\). To apply the standard dynamic programming and viscosity-solution results under
continuous data, we replace \(\psi\) by a Lipschitz continuous approximation. Let
\[
H_K(m):=\frac{1}{1+\exp(-Km)}, \qquad K>0,
\]
and define
\[
\psi_K(m)
:=
s_L
+
(s_M-s_L)H_K(m)
+
(s_H-s_M)H_K(m-a).
\]
For every fixed \(K<\infty\), the function \(\psi_K\) is bounded and Lipschitz continuous. Moreover,
\(\psi_K(m)\to \psi(m)\) as \(K\to\infty\) for every \(m\notin\{0,a\}\). In the open-economy problem,
the theoretical characterization is therefore stated for the regularized running cost obtained by
replacing \(\psi\) with \(\psi_K\).
\\

The results in the sequel rigorously hold for the value functions of the regularised problem, which clearly depend on the regularisation parameter $K>0$. However, to simplify notation, we leave the dependence on \(K\) implicit.

\begin{theorem}\label{thm:dyn_prog_eqn}
The value functions $\{v^i, i\in \mathbb I\}$ of the regularized problem are  the unique continuous viscosity solutions to the following  system of variational inequalities
\begin{subequations}
\label{eqn:HJB_BC}
\begin{align}
 \max \left\{ -\left( \partial_t v^i + \mathcal{L}^i v^i \right) - f(x), v^i - \min_{  j \neq i} (v^j + c_{i,j}) \right\} = 0, &\qquad \text{on } [0,T)\times Q, i \in \mathbb{I}
 \label{eqn:HJB} \\
 v^i = 0 &\qquad\text{on } \{T\}\times Q, i \in \mathbb{I}.
 \label{boundary_conditions}
\end{align}
\end{subequations}
\end{theorem}

\begin{proof} 
See Appendix \ref{sec:proofs} for a sketch.
\end{proof}

\section{Semi-Lagrangian scheme}\label{sec:numerics}
In what follows we consider a generic state dimension $d$. We will have $d=2$ in the context of a closed economy and $d=3$ in the open economy case. Also, we will work in a general framework where $f$ and $g$ allowed to depend on the regime, writing $f^i(\cdot):=f(\cdot, i)$ and $g^i(\cdot):=g(\cdot, i)$, and the volatility $\sigma: [0,T]\times \R^d\times \mathbb I\to \R^{d\times d}$ is not necesserely constant. 
Let $\mathcal{T} = \{ t_n = n\Delta t, \ n = 0,\dots,N \}$ define an even partition of the interval $[0,T]$ with mesh size $\Delta t = \frac{T}{N}$ and $\mathcal G = \{x_k, k\in \mathbb K\}$ a grid on the state space $Q$ with mesh size $|\Delta x|$. Fix $n\in \{0,\ldots, N-1\}$, the regime $i \in \mathbb{I}$ and the  state $x\in \mathcal G$ at time $t_n$.
We assume that over a time interval $[t_n,t_{n+1}]$,  one only has two possible actions:
\begin{enumerate}
    \item[C.] Continue with regime $i$ until $t_{n+1}$: this implies to pay the running costs associated with regime $i$ for the entire interval $[t_n,t_{n+1}]$, and keep the option to switch later. This is worth 
    \begin{equation*}
       C^i_n(x):= \mathbb{E} \left[ \int_{t_n}^{t_{n+1}} f_s^i(X^{(i)}_s) ds + v^i(t_{n+1},X^{(i)}_{t_{n+1}}) \middle| X_{t_n} = x\right],
    \end{equation*}
    where notation $(X_t^{(i)})_{t\in [t_n,T]}$ indicates that the process $X$ has evolved according to the dynamics implied by regime $i$.
    \item[S.] Switch immediately to some $j \neq i$: this implies to pay the switching cost $c_{i,j}$ at time $t_n$ and enjoy the value $v^j(t_n, x)$ from that point on. Optimising among all regimes $j \in \mathbb I\setminus\{i\}$ is worth the following: 
    \begin{equation*}
        S^i_n(x) :=\min_{j \neq i} (v^j(t_n,x)+c_{i,j}).
    \end{equation*}
\end{enumerate}
Under this assumption, $v(t_n, x) = \min \left\{ 
 C^i_n(x),
S^i_n(x)
\right\}$.
\\

We are now going to define an easily computable approximation of the continuation value $C^i_n(x)$. First, observe that the differential dynamics in \eqref{eqn:X_dyn} can be approximated by means of the Euler-Maruyama scheme. In particular, given $\varepsilon \sim \mathcal{N}(0,I)$,  $X_{t_{n+1}}^{(i)}$ is approximated by
\begin{equation*}
X_{n+1} = x + b(t_n,x,i) \Delta t + \sigma(t_n,x,i) \sqrt{\Delta t} \ \varepsilon.    
\end{equation*}
Given a smooth test function $\varphi$, a Taylor expansion around $x$ gives 
\begin{equation*}
    \mathbb{E} [\varphi(X_{n+1})] = \varphi(x) +  \mathcal{L}^i \varphi(x) \Delta t + o(\Delta t), 
\end{equation*}
thus crucially implying that only the first two moments of the innovation (i.e. mean and variance) matter at order one in $\Delta t$.
 More specifically, if  $h := b(t_n, x, i) \Delta t + \sigma(t_n, x, i) \sqrt{\Delta t} \ \varepsilon$ one has 
\begin{align*}
\mathbb{E}[h] = b(t_n,x,i) \Delta t, 
    \qquad  \mathbb{E}[h^{\otimes 2}] = (\sigma\sigma^\top) (t_n,x,i) \Delta t  + o(\Delta t)
    \quad\text{and}\quad \mathbb{E}[h^{\otimes k}] = o(\Delta t) \quad (k\geq 3).
\end{align*}
We consequently replace $\varepsilon$ with $\xi \sim \text{Rad}^d:= \{-1,1\}^d$ and introduce the auxiliary variable
\begin{equation}\label{eqn:rv2footpoints}
    \tilde X_{n+1} =  x + b(t_n,x,i) \Delta t + \sigma(t_n,x,i) \sqrt{\Delta t} \ \xi
\end{equation}
as driven by the simplest possible engine so that $\mathbb{E}[\varepsilon] = \mathbb{E}[\xi]$ and $\mathbb{E}[\varepsilon^{\otimes 2}] = \mathbb{E}[\xi^{\otimes 2}]$. \\
By construction, we therefore have
\begin{equation*}
    \mathbb{E} [\varphi(\tilde X_{n+1})] = \varphi(x) +  \mathcal{L}^i \varphi(x)\Delta t + o(\Delta t). 
\end{equation*}
A more convenient expression for the latter is now given as a finite sum of evaluations of function $\varphi$ at footpoints, i.e.

\begin{equation*}
    \mathbb{E} [\varphi(\tilde X_{n+1})] = \frac{1}{2^d} \sum_{s \in \{-1,+1\}^d} \varphi(p_s(x_k))
\end{equation*}
with
\begin{equation*}
    p_{s_1,\dots,s_d} = x + b(t_n,x,i) \Delta t + \left[ \sum_{k=1}^d (s_1 \sigma_{k1} + \dots + s_d \sigma_{kd})(t_n,x,i) e_k \right] \sqrt{\Delta t}
\end{equation*}
for full generality.

Then,  $v(t_n, x)$ can be replaced by the following approximation of order $\Delta t$ in time,
\begin{equation}\label{eq:approx}
v^i_n(x) = 
\min \left\{ 
\tilde C^i_n(x),
\tilde S^i_n(x)
\right\}.
\end{equation} 
where 
\begin{align*}
    \tilde C^i_n(x) & := f^i(t_n,x) \Delta t + \frac{1}{2^d} \sum_{l=1}^{2^d}  \mathfrak{I}[v^i_{n+1}](p_l(x))\\
    \tilde S^i_n(x) & := \min_{j \neq i} (v^j_n(x)+c_{i,j})
\end{align*}
where $\mathfrak{I}[\cdot]$ is any $d$-dimensional monotone interpolator (i.e. non-negative weights). The operator $\mathfrak{I}[\cdot]$ is needed to provide approximated values of $v^i_{n+1}$ at footpoints $(p_l(x))$ that are possibly not on $\mathcal G$ even if $x$ is, while the monotonicity requirement is related to theoretical property of the algorithm ensuring convergence. Since monotone interpolator are only of first order, the expected accuracy in space of the present approximation is $|\Delta x|$.

% \textbf{The continuation region} \\

% % Let $\varphi(x) = v^i(t_k+\Delta,x)$. 
% % We eventually conclude with the following approximation of the continuation value at order one in $\Delta$
% % \begin{equation*}
% %     \mathbb{E} \left[
% % \int_{t_k}^{t_{k+1}} f_s^i(X^{(i)}_s) ds + v^i(t_k+\Delta,X^{(i)}_{t_k+\Delta}) \middle| X_{t_k} = x_k \right] = f_{t_k}^i(x_k) \Delta + \frac{1}{2^d} \sum_{l=1}^{2^d} \mathfrak{I} \left[ v^i \right] (t_k+\Delta,p_l(x_k)),
% % \end{equation*}
% %  \\

%\textbf{The switching iteration} \\

The algorithm we present exploits \eqref{eq:approx} to compute  backward from terminal time $T$ to $t_0 = 0$ an approximation of the value function. Each step in time requires to solve a switching problem. 
Let $n\in \{0,\ldots N-1\}$, $k\in \mathbb K$, $v= [v^1_n \dots v^{|\mathbb I|}_n]^\top (x_k)$ and $\Gamma: \mathbb{R}^{|\mathbb I|} \to \mathbb{R}^{|\mathbb I|}$ be defined as 
\begin{equation}\label{eqn:fixed_point_map}
(\Gamma(v))_i := \min \left\{ \tilde C^i_n(x_k) , \tilde S^i_n(x_k) \right\}, \qquad i \in \mathbb{I}. 
\end{equation}
Notice that the dependence on $v$ of the operator $\Gamma$ comes from the term  $\tilde S^i_n(x_k)$. Equation \eqref{eq:approx} then defines the fixed point problem 
$$
v = \Gamma(v),
$$
which clearly calls for a fixed-point iteration
\begin{equation}\label{eqn:fixed_point_iter}
    \begin{aligned}
        & v^{(0)} = (\tilde C^1_n(x_k),\ldots,\tilde C^{|\mathbb I|}_n(x_k))^\top \\
        & v^{(q+1)} = \Gamma(v^{(q)}), \qquad \text{ for } q = 0,1,2, \dots
    \end{aligned}
\end{equation}
Observe that the fixed point concerns the switching part $\tilde S^i$ only; in fact, $v^i_{n+1}$ defining the continuation value $\tilde C^i_n$ is already known at $t_n$ when marching backward in time over $\mathcal{T}$.

% where
% \begin{align*}
%     C^i & = f_{t_k}^i(x_k) \Delta + \frac{1}{2^d} \sum_{l=1}^{2^d} \mathfrak{I} \left[ v^i \right](t_k+\Delta,p_l(x_k)) \\
%     S^i & = \min_{j \neq i} (v^j(t_k,x_k)+c_{i,j}).
% \end{align*}
% Now, 

\begin{proposition} The following properties hold:
\begin{itemize}
        \item[(i)] $\Gamma$ is monotone, i.e. if $v \leq w$, then $\Gamma(v) \leq \Gamma(w)$ (where inequalities are meant to hold componentwise);
        \item[(ii)] $\Gamma$ is $1$-Lipschitz continuous, i.e.
        \begin{equation*}
            \|\Gamma(v)-\Gamma(w)\|_\infty \leq \|v-w\|_\infty.
        \end{equation*}
       \item[(iii)] The iteration \eqref{eqn:fixed_point_iter} converges in one step to the fixed point 
       $$
       \hat v_i = \min_{j=1,\dots,|\mathbb I|} (\tilde C^j_n(x_k) + c_{i,j})
       $$
    \end{itemize}   
\end{proposition}

\begin{proof}
$(i)$ is a direct consequence of the definition of $\Gamma$. For any $n\in \{0,\ldots, N-1\}$, $k\in \mathbb K$, $i\in \mathbb I$, let us denote $\tilde C^i=\tilde C^i_n(x_k)$. One has,
\begin{align*}
\left|\Gamma(v)_i - \Gamma(w)_i\right|\leq \left|\min\left\{\tilde C^i, \min_{j \neq i} (v^j+c_{i,j})\right\} - \min \left\{\tilde C^i, \min_{j \neq i} (w^j+c_{i,j})\right\}\right| \leq \max_{j\neq i} |v^j - w^j|
\end{align*}
from which $(ii)$ immediately follows.
Let us now prove $(iii)$. We first provide the following explicit representation of the iterate $v^{(q)}$,
\begin{equation}\label{eq:explicit}
        v_i^{(q)} = \min_{j=1,\dots,|\mathbb I|} (\tilde C^j + c_{i,j}), \qquad q \geq 1, i\in \mathbb I.
\end{equation}
We proceed by induction on $q$. For $q = 1$, one has 
        \begin{equation*}
            v_i^{(1)} = \Gamma(v^{(0)})_i = \min(\tilde C^i, \min_{j\neq i} (v^{(0)}_j + c_{i,j})) = \min(\tilde C^i, \min_{j\neq i} (\tilde C^j + c_{i,j}))=\min_{j}(\tilde C^j + c_{i,j})
 \end{equation*}
 Moreover, assuming \eqref{eq:explicit} holds for  some $q\geq 1$, it follows
 \begin{align*}
 v_i^{(q+1)} &  = \Gamma(v^{(q)})_i = \min \left( \tilde C^i, \min_{j \neq i} (v_j^{(q)} + c_{i,j}) \right) 
         = \min \left\{ \tilde C^i, \min_{j \neq i}\left[ \min_l \left(\tilde C^l + c_{j,l}\right) + c_{i,j}\right] \right\} \\
            & = \min \left\{ \tilde C^i, \min_l\left[ \tilde C^l + \min_{j \neq i} \left( c_{j,l} + c_{i,j}\right) \right]\right\}= \min \left\{ \tilde C^i, \min_{l\neq i}\left[ \tilde C^l + \min_{j \neq i} \left( c_{j,l} + c_{i,j}\right) \right]\right\},
    \end{align*}         
 where the last equality is obtained observing that 
 $
 \tilde C^i < \tilde C^i + \min_{j \neq i} \left( c_{j,i} + c_{i,j}\right)
 $. The result then follows since, if $l\neq i$,  $\min_{j \neq i} \left( c_{j,l} + c_{i,j}\right) = c_{l,l} + c_{i,l} = c_{i,l}.$
Since the right hand side of \eqref{eq:explicit} does not depend on $q$ it is clear that it represents the fixed point, that is then reached in one iteration.          
            
\end{proof}
% \begin{remark}
%     Clearly the closed-form formula \eqref{eq:explicit} can  replace the switching iteration for fast computation when $|\mathbb I|$ is large.    
% \end{remark}

\textbf{Algorithm} \\
\begin{algorithm}[H]
\caption{Semi-Lagrangian scheme for finite horizon optimal switching problems}
\begin{algorithmic}[1]
\State Define $d$-dim lattice $\mathcal{G}$ for the space variables $X=(x_1,\dots,x_{\mathbb K})$
\State Initialize value table
\Statex \quad $v$ = zeros($N+1$,$|\mathbb K|$,$|\mathbb I|$) 
\State Apply terminal conditions
\Statex \algorithmicfor\ $i = 1:|\mathbb{I}|$ \algorithmicdo
    \Statex \quad $v$($N+1$,\ :\ ,$i$) = $g^i(:)$ 
\Statex \algorithmicend\ \algorithmicfor
\State Backward recursion
\Statex \algorithmicfor\ $n = N:-1:1$ \algorithmicdo
    \Statex \quad \% Build continuation candidates 
    \Statex \quad $C$ = zeros($|\mathbb K|$,$|\mathbb I|$)
    \Statex \quad \algorithmicfor\ $i = 1:|\mathbb I|$ \algorithmicdo
    \Statex \quad \quad $C$(\ :\ ,$i$) = $f^i(t_n,\ :\ ) \Delta t + \frac{1}{2^d} \sum_{l=1}^{2^d} \mathfrak{I} \left[ v(n+1,\ :\ , i )\right](p_l(\ :\ ) )$
    \Statex \quad \algorithmicend\ \algorithmicfor
    \Statex \quad \% Solve switching iteration
   % \Statex \quad $U$ = $C$
   % \Statex \quad \algorithmicfor\ $q = 1:|\mathbb I|-1$ \algorithmicdo
   % \Statex \quad \quad $U$ = $\Gamma(U)$
   % \Statex \quad \algorithmicend\ \algorithmicfor
    \Statex \quad $v(n,:,:)$ = $\Gamma(C)$
\Statex \algorithmicend\ \algorithmicfor
\State (optional) Extract policy table
\end{algorithmic}
\end{algorithm}

\section{Numerical experiments}\label{sec:results}

This section illustrates how the optimal load-following policy depends on residual-demand uncertainty, shortage costs, switching costs, ramping capability, and market access. Matlab codes for all numerical experiments are available at \url{https://github.com/fabioBaschetti/OSNPP}. The Italian data are used to calibrate a representative high-frequency residual-demand process. The nuclear operating parameters are normalized and stylized; they are not intended to describe the current Italian generation mix.
Residual demand is constructed as total load net of renewable production and normalized by a reference capacity level, so that the state variables in the numerical model are expressed in dimensionless units.\\

We fix $P^{\text{min}} = 0.2$ and $P^{\text{max}} = 0.9$ for the operations of the nuclear power plant. Estimated parameters will always be denoted by hats, while barred quantities denote reference values used in the comparative statics.\\

We consider an Ornstein-Uhlenbeck process for the residual demand as described in Equation \eqref{eqn:res_dem_OU}, where the time-varying mean-reversion level $\theta_t$ includes nine seasonal components with periods $(\frac{1}{4},\frac{1}{3},\frac{1}{2},1)$ days, $(\frac{1}{2},1)$ weeks, and $(\frac{1}{4},\frac{1}{2},1)$ years. 
Model parameters are estimated using the least-squares procedure described in Appendix~\ref{sec:estimation}:
\begin{align*}
    & \hat{\kappa} = 0.3500 \qquad \hat{\beta} = 0.6118 \qquad \hat{\nu} = 0.1114 \\
    & \hat{\zeta} = 
    \begin{bmatrix}
        \hspace{\widthof{$-$}}0.4100 & \hspace{\widthof{$-$}}0.1606 & -2.4238 & -1.5101 & \hspace{\widthof{$-$}}0.0841 & \hspace{\widthof{$-$}}0.2984 & -0.0113 & \hspace{\widthof{$-$}}0.0563 & \hspace{\widthof{$-$}}0.0912
    \end{bmatrix} \\
    & \hat{\eta} = 
    \begin{bmatrix}
         \hspace{\widthof{$-$}}0.2714 & -0.6401 & \hspace{\widthof{$-$}}2.8156 & -0.9522 & -0.2479 & \hspace{\widthof{$-$}}0.0982 & -0.0162 & \hspace{\widthof{$-$}}0.0451 & -0.0527
    \end{bmatrix}   
\end{align*}
High-frequency (15-minute) data (for both the total load and the renewable production) ranging from 01/01/2024 to 31/12/2025 are publicly available for the Italian market at Terna's Download Center\footnote{\url{https://dati.terna.it/download-center}}. \\

Moreover, we assume a 5\% ramping rate per unit time (i.e., $\bar r  = 0.05 / \Delta t$) and introduce the following switching costs:
\begin{align*}
    \bar C
    = 
    \begin{bmatrix}
        0 & 4 & 7 \\
        1.6 & 0 & 4.8 \\
        1.6 & 0.4 & 0
    \end{bmatrix}
    10^{-4}.
\end{align*}
Observe that switching to a higher production regime is more expensive than switching to a lower state. However, we want to discourage excessively low production as a way to amortize the financial costs of the nuclear power plant (e.g. by $c_{+r, -r} > c_{+r, 0}$). \\

We extend our optimal switching problem up to 1 year, but mainly plot short horizons (e.g. 1 week or 1 month) for visualization purposes. 

\subsection{Closed economy}

As for the closed economy case, we set
\begin{align*}
 \gamma_1 = 0.24 \qquad 
 \bar \lambda_{1} = 0 \qquad
 \bar \lambda_{2} = 0.48 
\end{align*}
and observe the following:
\begin{remark}
One can always rescale the coefficients in order to assume no  operating costs other than the switching decisions, i.e. taking $\gamma_1=0$. This can be done without loss of generality. Indeed, for any $\gamma_1\geq 0$ one has  
\begin{align*}
f((p,y)) = \lambda_1 (p - y)^+ + \lambda_2 (y - p)^+ + \gamma_1 (p-y) + \gamma_1 y =  (\lambda_1+\gamma_1) (p - y)^+ + (\lambda_2-\gamma_1) (y - p)^+ + \gamma_1 y
\end{align*}
and, being the term $\gamma_1 y$ uncontrolled, the optimal strategy can be directly computed considering no operating cost and shifted parameters $(\lambda_1+\gamma_1)$ and $(\lambda_2 - \gamma_1)$.
\end{remark}
Because of this flexibility on $\gamma_1$, there is no need for us to study the dependence of the control on its actual value (this is why it comes with no bar). \\

We solve for the value functions using the semi-Lagrangian scheme described above and extract the
associated optimal policies. For each time \(t\in[0,T]\), regime \(i\in\mathbb I\), and state
\(x=(p,y)\), the policy table identifies the action that minimizes the continuation value conditional
on the current state.  \\

The following Figure \ref{fig:policy_table} provides an explicit example for the particular case where $t=0$ and $i_{0^-} = 0$.
\begin{figure}[H]
    \centering
    \includegraphics[width=0.6\linewidth]{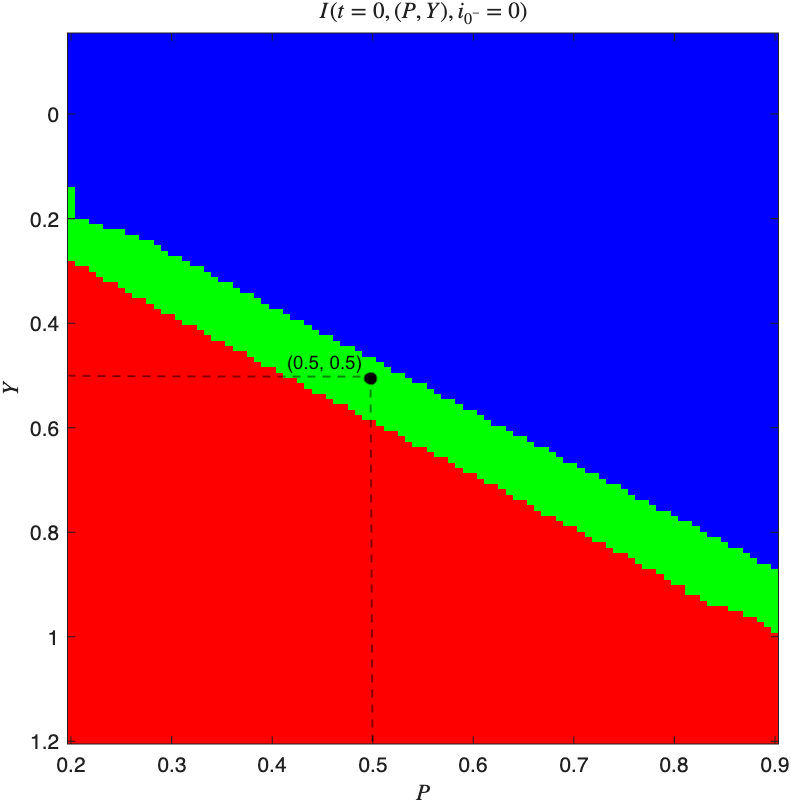}
    %{figures/policy_table.eps}
    \caption{Policy table at time $t=0$ and regime $i_{0^-}=0$. Green: $\to 0$, red: $\to +r$, blue: $\to -r$.}
    \label{fig:policy_table}
\end{figure}
Each pixel corresponds to an action: green, red, and blue indicate that the optimal post-switching
regime is \(0\), \(+r\), and \(-r\), respectively. For example, at the initial time \(t=0\), the point
\(x=(0.5,0.5)\) lies in the green region, so the optimal action is to keep production constant. 
Holding nuclear output fixed and increasing residual demand moves the state into the red region,
where it becomes optimal to increase production. Conversely, reducing residual demand moves the
state into the blue region, where it becomes optimal to decrease production. Similarly, for a fixed
level of residual demand, lower values of \(P\) favor the regime \(+r\), whereas higher values of \(P\)
favor the regime \(-r\).\\

The same interpretation applies at other times and for other initial regimes. The policy tables therefore
provide a state-dependent description of the optimal switching rule over the simulation horizon. \\

We therefore enforce the optimal policy throughout simulation of process $X=(P,Y)$ in Figure \ref{fig:simulation_x_p_1w}. When doing so, we confine ourselves to a short horizon -- of one week -- in the interest of readability. \\ 

Figure~4 shows that the optimal policy generates a close, but imperfect, alignment between nuclear
production and residual demand. When residual demand exceeds nuclear production, the remaining
gap must be covered by the costly back-up technology represented by the shortage penalty in the
objective function. This can occur either because residual demand exceeds installed nuclear capacity
or because, given switching and operating costs, immediate adjustment is not optimal. As shown in
the \(\lambda_2\)-sensitivity analysis below, increasing the shortage penalty induces the plant to track
residual demand more aggressively. 

%The effect of favoring over-production rather then under-production is now evident, with the nuclear output usually exceeding the residual demand. %Figure \ref{subfig:sim_p} reports the optimal switching strategy producing the best fit of $P$ to $Y$ (when starting from regime $i=2$).  

\begin{figure}[H]
    \centering    
    %    \begin{subfigure}{0.49\textwidth}
    %     \centering
    %     \includegraphics[width=\linewidth]{figures/sim_p_1w.eps}
    %     \caption{Optimal control}
    %     \label{subfig:sim_p}
    % \end{subfigure}
   \begin{subfigure}{0.49\textwidth}
        \centering
         \includegraphics[width=\linewidth]{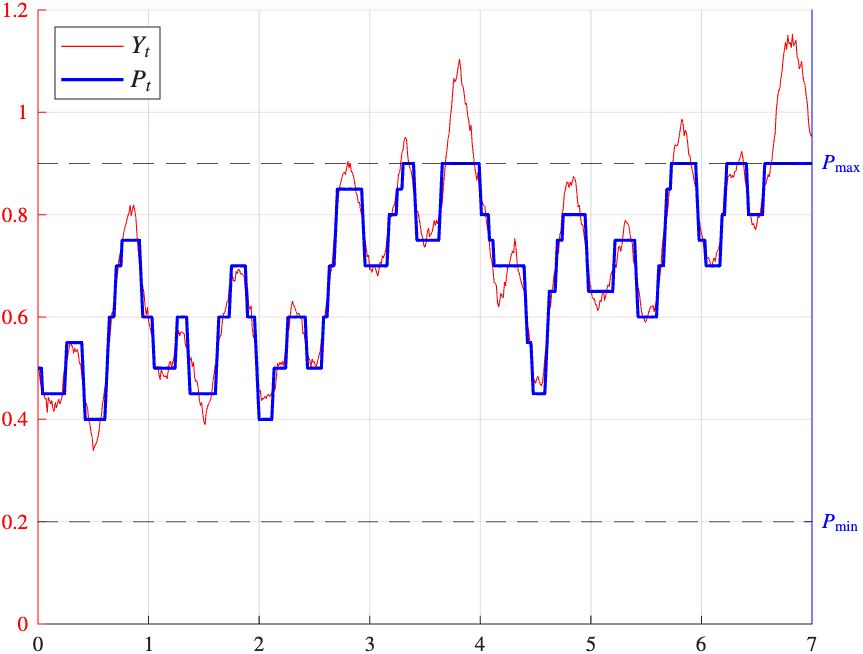}
       % \caption{Realized dynamics of $X$}
        \label{subfig:sim_x}
    \end{subfigure}
    \hfill
     \begin{subfigure}{0.49\textwidth}
         \centering
         \includegraphics[width=\linewidth]{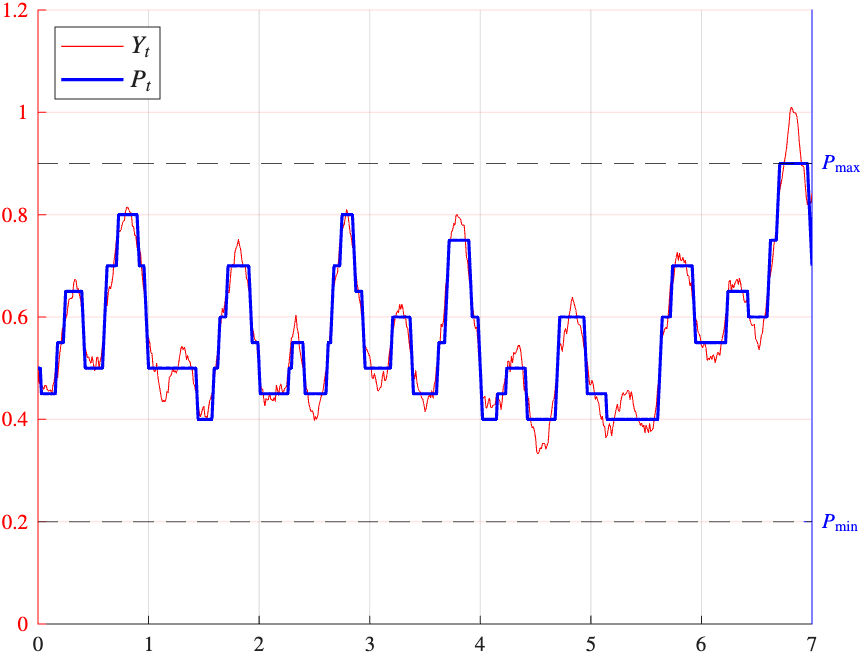}
     %\caption{Optimal control}
         \label{subfig:sim_p}
     \end{subfigure}
    \caption{Two different realizations of the residual demand $\{Y_t\}_{t\in [0,T]}$ (red line) and the associated optimal production rate $\{P_t\}_{t\in [0,T]}$ (blue line) in a simulation environment over one week.}
    \label{fig:simulation_x_p_1w}
\end{figure}

Figures \ref{fig:simulation_x_1m} and \ref{fig:simulation_x_1y}  extend the time horizon to one month and one year, respectively.

\begin{figure}[H]
    \centering
    \includegraphics[width=0.8\linewidth]{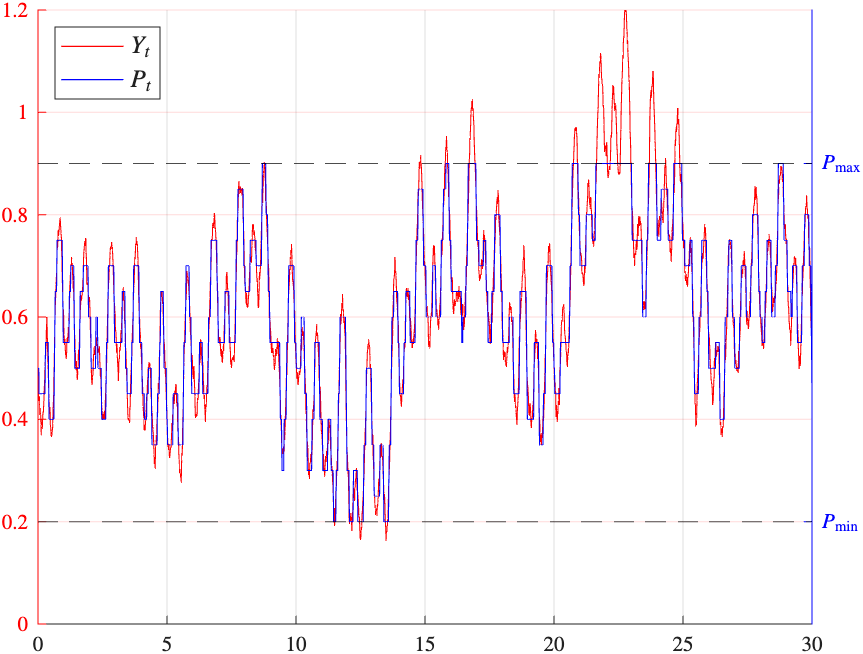}
    \caption{The optimal switching strategy in a simulation environment over one month (30 days).}
    \label{fig:simulation_x_1m}
\end{figure}
\begin{figure}[H]
    \centering
    \includegraphics[width=0.8\linewidth]{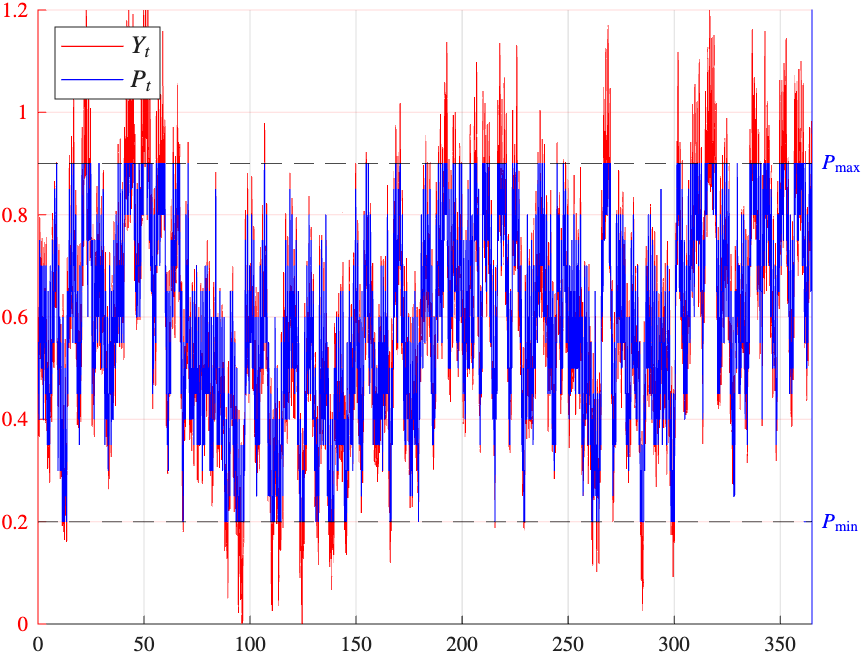}
    \caption{The optimal switching strategy in a simulation environment over one year (365 days).}
    \label{fig:simulation_x_1y}
\end{figure}

\subsubsection{Sensitivity analysis}

%\subsubsubsection{The cost of under-production}

Parameter $\lambda_2$ regulates the cost of under-production: the higher $\lambda_2$, the stronger the constraint to fully meet the residual demand. Figure \ref{fig:simulation_x_1w_L2} shows that the optimal policy reacts precisely as expected. More specifically, we can see that no blackout happens when $\lambda_2$ is large enough (unless the demand exceeds the nominal capacity of the nuclear power plant).   

\begin{figure}[H]
    \centering    
    \begin{subfigure}{0.49\textwidth}
        \centering
        \includegraphics[width=\linewidth]{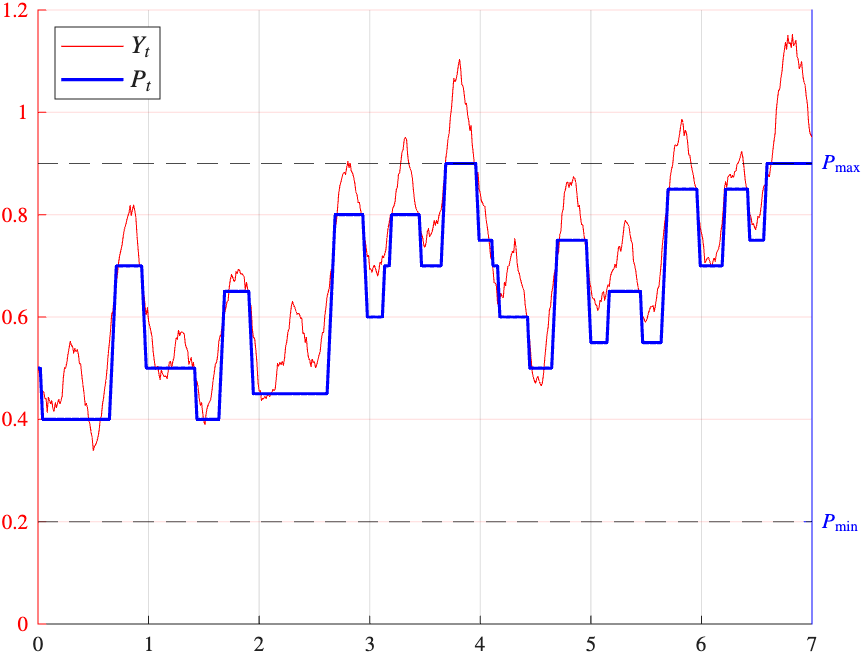}
        \caption{$\lambda_2 = 0.6 \bar\lambda_2$}
        \label{subfig:sim_x_lowL2}
    \end{subfigure}
    \hfill
    \begin{subfigure}{0.49\textwidth}
        \centering
        \includegraphics[width=\linewidth]{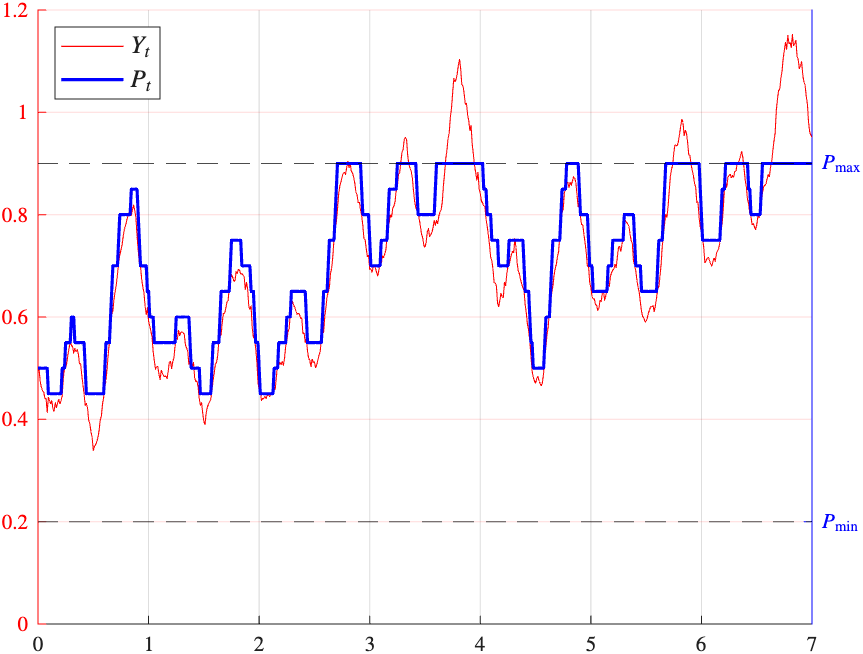}
        \caption{$\lambda_2 = 10 \bar\lambda_2$}
        \label{subfig:sim_x_higL2}
    \end{subfigure}
    \caption{The optimal switching strategy as changing the cost of under-production. Simulation over one week.}
    \label{fig:simulation_x_1w_L2}
\end{figure}

%\subsubsection{Changing the switching costs}

Matrix \(C\) summarizes the costs of switching from regime \(i\) to regime \(j\). Increasing all
switching costs through a common multiplicative factor makes regime changes less attractive.
Figure~\ref{fig:simulation_x_1w_G} illustrates this effect: higher switching costs reduce the frequency of adjustments and
generate smoother production profiles.

\begin{figure}[H]
    \centering    
    \begin{subfigure}{0.49\textwidth}
        \centering
        \includegraphics[width=\linewidth]{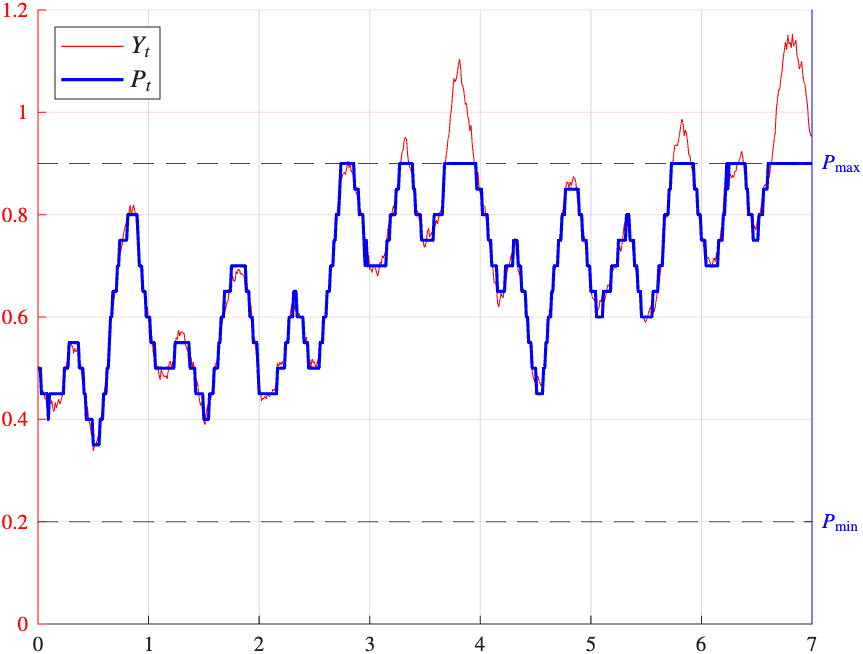}
        \caption{$C = 0.1 \bar C$}
        \label{subfig:sim_x_lowG}
    \end{subfigure}
    \hfill
    \begin{subfigure}{0.49\textwidth}
        \centering
        \includegraphics[width=\linewidth]{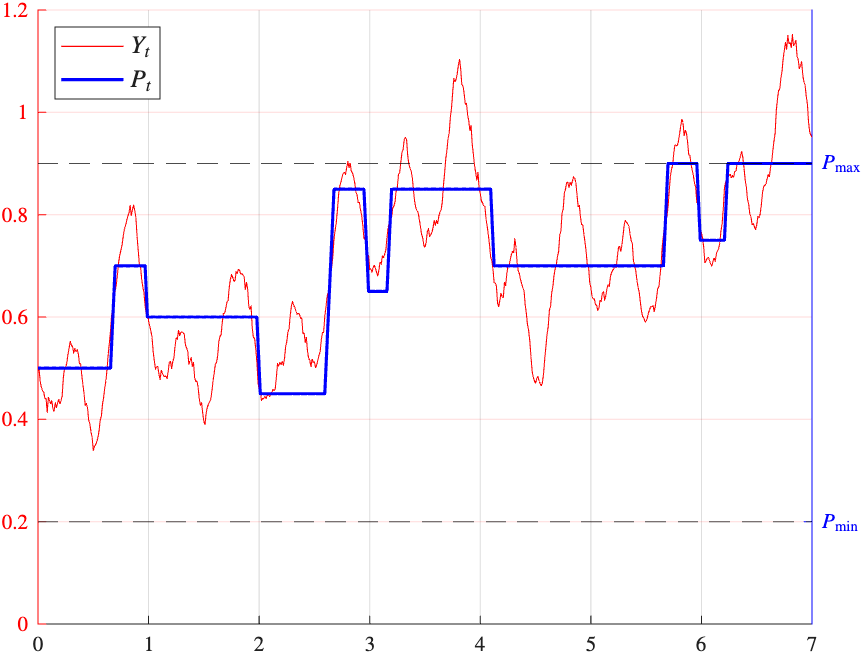}
        \caption{$C = 10 \bar C$}
        \label{subfig:sim_x_higG}
    \end{subfigure}
    \caption{The optimal switching strategy as changing the switching costs. Simulation over one week.}
    \label{fig:simulation_x_1w_G}
\end{figure}

The ramping rate \(r\) controls the maneuverability of the nuclear power plant: higher values of \(r\)
allow nuclear output to adjust more rapidly to changes in residual demand. Figure~\ref{fig:simulation_x_1w_r} highlights an
important feature of the discrete-regime formulation. 
Recall that $r$ is a fixed rate that cannot be optimized in our setting. Therefore increasing it does not necessarily improve tracking.
Large fixed ramping rates may generate overshooting and additional regime changes when smaller adjustments
would be sufficient. In this sense, a lower fixed ramping rate can therefore produce smoother
tracking than a larger one.  

\begin{figure}[H]
    \centering    
    \begin{subfigure}{0.49\textwidth}
        \centering
        \includegraphics[width=\linewidth]{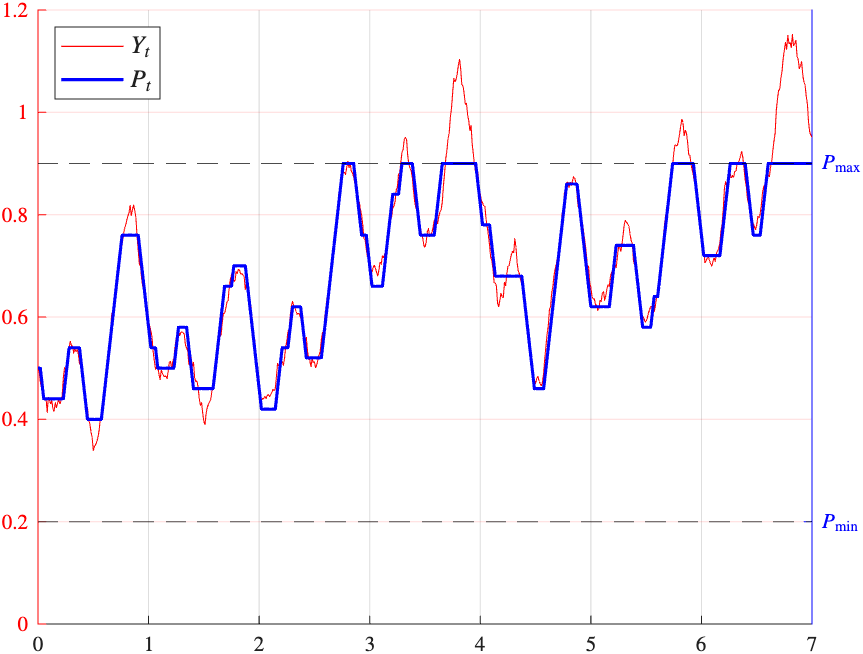}
        \caption{$r = 0.4 \bar r$}
        \label{subfig:sim_x_lowr}
    \end{subfigure}
    \hfill
    \begin{subfigure}{0.49\textwidth}
        \centering
        \includegraphics[width=\linewidth]{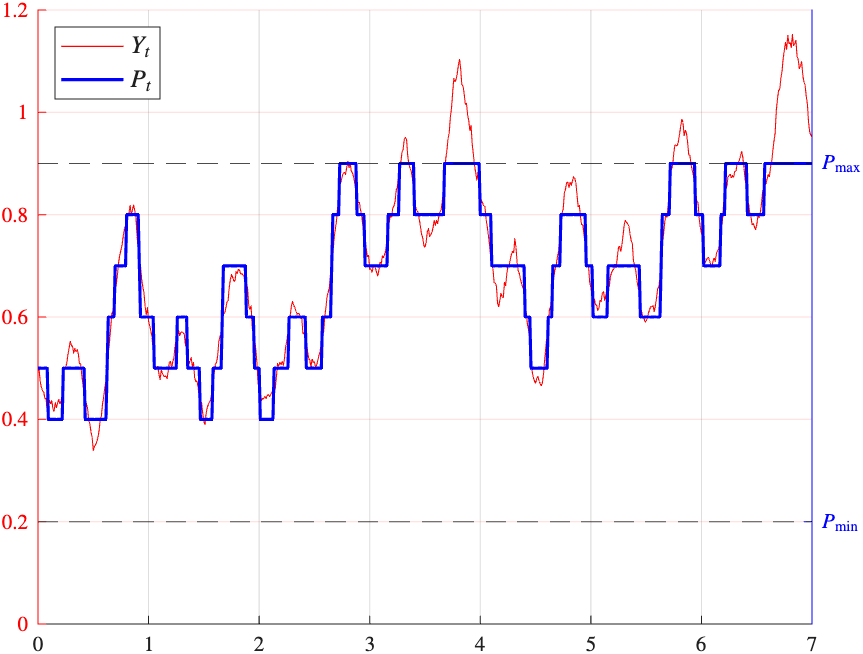}
        \caption{$r = 2\bar r$}
        \label{subfig:sim_x_higr}
    \end{subfigure}
    \caption{The optimal switching strategy as changing the ramping rate. Simulation over one week.}
    \label{fig:simulation_x_1w_r}
\end{figure}

\subsection{Open economy}
We now consider the case where the economy is open. The residual demand at the local level is the same as before. At the "global" level (i.e., for the whole electricity market), we assume that the residual demand $M$ is obtained by rescaling
\begin{align*}
    & \kappa_M = \kappa_Y \qquad \beta_M = nl\beta_Y \qquad \nu_M = \sqrt{n (1+(n-1)\rho)}\nu_Y \\
    & \zeta_M = n\zeta_Y \\
    & \eta_M = n\eta_Y
\end{align*}
where $n$ is the number of countries participating to the market and $\rho$ is the correlation between Brownian motions $W^Y$ and $W^M$. Finally, $l > 0$ is a constant that fixes the average offset of $M$ with respect to the average offset of $Y$.
We also add a temporal shift $s$ accounting for the geographical differences between one single country and the system. \\

For our numerical experiments, we fix
\begin{equation*}
    n = 5 \qquad \rho = 0 \qquad l = 0.4 \qquad s = 7 \ \text{days}
\end{equation*}
thus resulting in the following Figure \ref{fig:resdem_IT_vs_EU}.
\begin{figure}[H]
    \centering
    \includegraphics[width=0.75\linewidth]{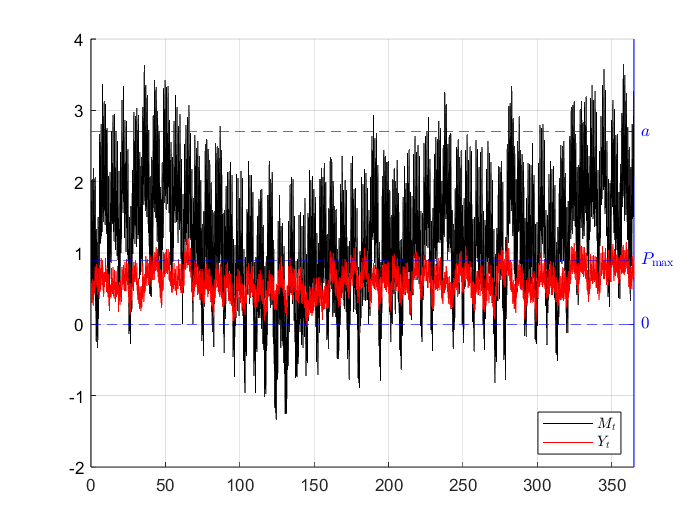}
    \caption{Simulated paths for the residual demand at the local ($Y$) and the global ($M$) level.}
    \label{fig:resdem_IT_vs_EU}
\end{figure}

The figure also clarifies that the global nuclear capacity is fixed at $a  = n L P^{\text{max}}$, with $L=0.6$. This reflects the past investment decisions in nuclear technologies from the countries involved in the system. \\

The interpretation for the setting described above is as follows. We assume that the market residual demand $M$ and the local residual demand $Y$ share the same speed of mean reversion, implying that deviations from their long-run seasonal levels dissipate at the same rate. This reflects the idea that both the local system and the aggregate market are subject to the same persistence in the underlying demand and renewable-generation dynamics. The global baseline level $\beta_M$ scales linearly with the number of market participants $n$, while the additional multiplicative factor $l \neq 1$ captures structural heterogeneity across countries. In particular, it summarizes differences in the endowment of renewable resources, installed renewable capacity, and generation mixes, which translate into different long-run average levels of residual demand. It follows that such a structural heterogeneity does not affect the (strength of the) deterministic seasonal cycle. Random fluctuations naturally scale with $\sqrt{n}$. We assume zero correlation between the driving Brownian motions under the idea that the noise is only dictated by local factors. Finally, a temporal shift introduces a delay in the seasonal cycle as justified by spatial factors (i.e., latitude and longitude). \\

\textbf{Price taker} \\

We start by assuming the company to be a \emph{price taker}. We allow for three different price scenarios 
$$
s_{L}=0,\qquad s_{M}=0.2, \qquad s_{H} = 0.4,
$$
corresponding to the market clearing resource being renewable/nuclear/fossil power, respectively. The bid-ask spread on that price is $\delta = 0.08$. Then, 
$\gamma_1 = s_M + 0.5 \delta = 0.24 $ recovers the same operating costs as it was in the closed-economy case. Observe that the hierarchy $s_M< \gamma_1< s_M+\delta$ ensures that the local company will never buy from the market at price $s_M+\delta$ if it can meet the residual demand by producing internally (at $\gamma_1$). Similarly, it has no incentive in over-producing at $\gamma_1$ if it can sell at $s_M$. \\

As detailed in Section \ref{sec:open} the selling price is determined by the residual demand on the market via the function $\psi(\cdot)$. Figure \ref{fig:price} shows the determination of the price scenario starting from different realization of the process $\{M_t\}_{t\in [0,T]}$.  
\begin{figure}[H]
    \centering    
    \begin{subfigure}{0.49\textwidth}
        \centering
         \includegraphics[width=\linewidth]{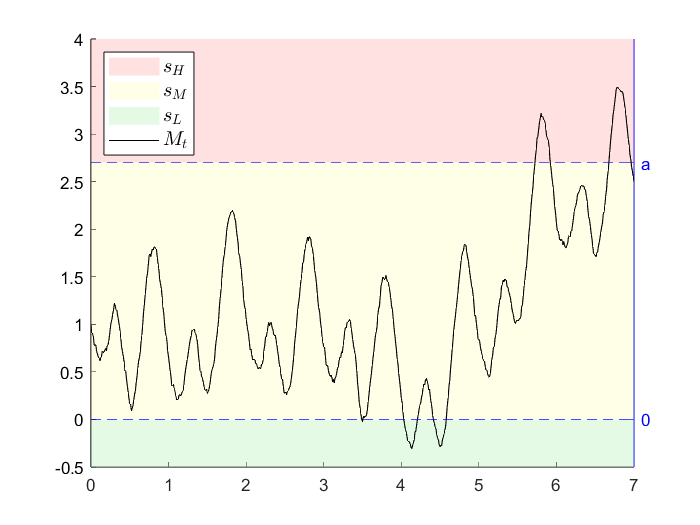}
       % \caption{$r = \frac{1}{5}\tilde r$}
        \label{subfig:M1}
    \end{subfigure}
    \hfill
    \begin{subfigure}{0.49\textwidth}
        \centering
        \includegraphics[width=\linewidth]{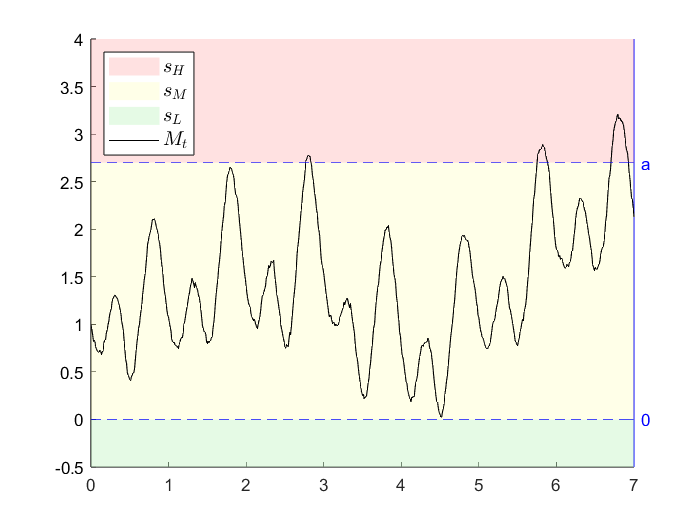}
        %\caption{$r = 2\tilde r$}
        \label{subfig:M2}
    \end{subfigure}
        \caption{Price determination: when the process $\{M_t\}_{t\in [0,T]}$ goes above the threshold $a$ electricity price is high (red area), when it goes below $0$ the price is low (green area). Two different realizations of $\{M_t\}_{t\in [0,T]}$ are represented.}
\label{fig:price}
    \end{figure}
Comparing Figure \ref{fig:simulation_x_1w_3d} with Figure \ref {fig:simulation_x_p_1w} underlines the benefits of an open economy. Indeed, starting from the same realization of $\{Y_t\}_{t\in [0,T]}$ it might now be optimal to underproduce when the price is low (left plot, green areas) and to overproduce when the market price is high (left and right plots, red areas).
\begin{figure}[H]
    \centering    
    \begin{subfigure}{0.49\textwidth}
        \centering
         \includegraphics[width=\linewidth]{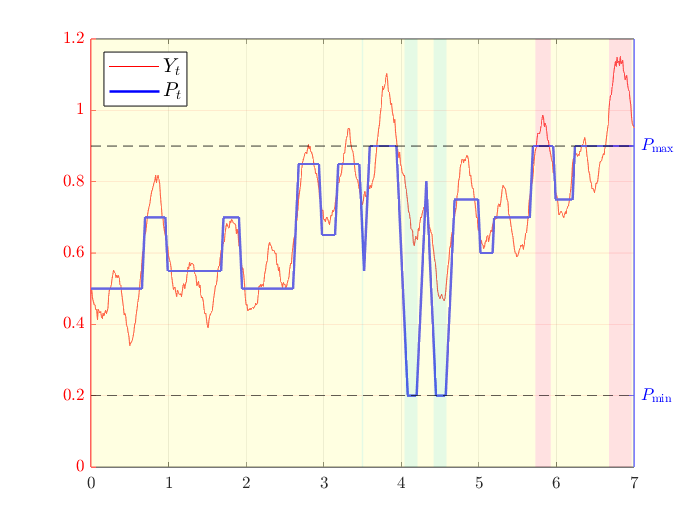}
       % \caption{$r = \frac{1}{5}\tilde r$}
        \label{subfig:1}
    \end{subfigure}
    \hfill
    \begin{subfigure}{0.49\textwidth}
        \centering
        \includegraphics[width=\linewidth]{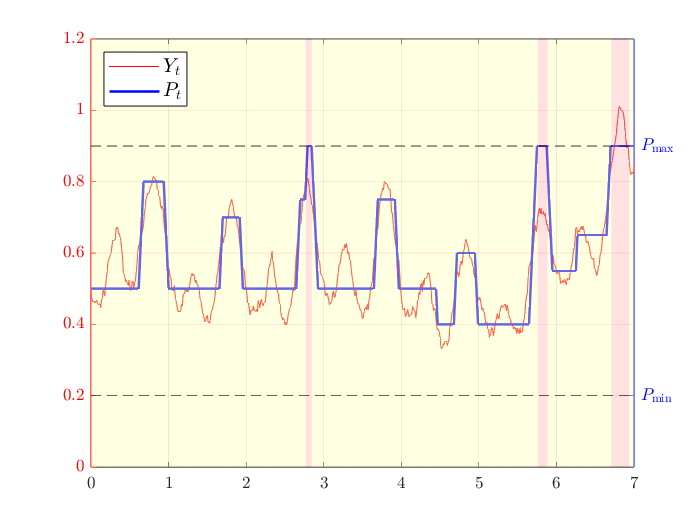}
        %\caption{$r = 2\tilde r$}
        \label{subfig:2}
    \end{subfigure}
    \caption{Two different realizations of the residual demand $\{Y_t\}_{t\in [0,T]}$ (red line) and the associated optimal production rate $\{P_t\}_{t\in [0,T]}$ (blue line). The red and green shaded regions indicate periods with high and low electricity prices, respectively,
while the yellow shaded region corresponds to the medium-price regime.}
    \label{fig:simulation_x_1w_3d}
\end{figure}

\textbf{Price maker} \\

We now consider the case where the company is a \textit{price maker}, meaning that its decisions directly influence the market price of electricity. \\

Figure \ref{fig:price_Z} is the exact analogue of Figure \ref{fig:price} in a context where the firm's production imbalance $q_t$ contributes to the market demand $M_t$.

\begin{figure}[H]
    \centering    
    \begin{subfigure}{0.49\textwidth}
        \centering
         \includegraphics[width=\linewidth]{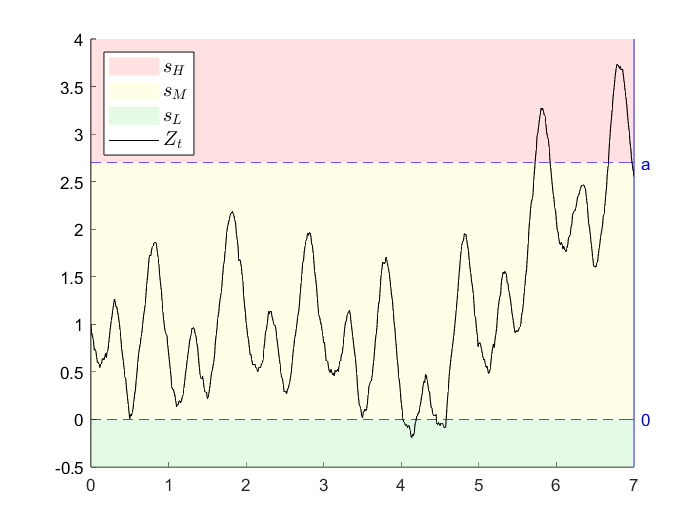}
        \label{subfig:Z1}
    \end{subfigure}
    \hfill
    \begin{subfigure}{0.49\textwidth}
        \centering
        \includegraphics[width=\linewidth]{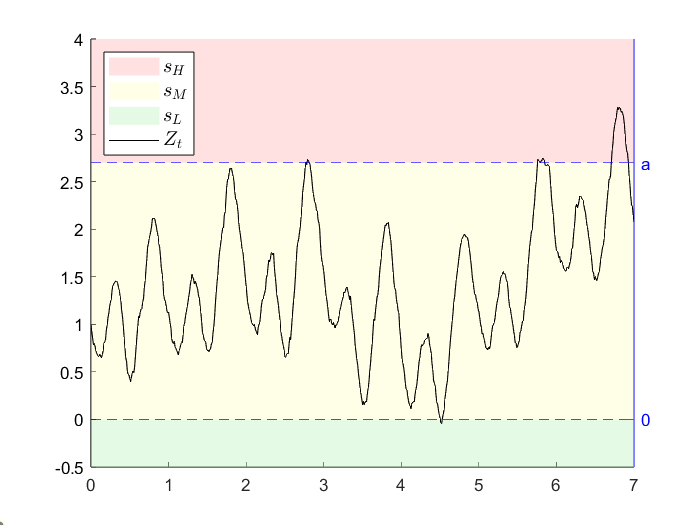}
        \label{subfig:Z2}
    \end{subfigure}
        \caption{Price determination: when the process $\{Z_t\}_{t\in [0,T]}$ goes above the threshold $a$ electricity price is high (red area), when it goes below $0$ the price is low (green area). Two different realizations of $\{Z_t\}_{t\in [0,T]}$ are represented.}
\label{fig:price_Z}
\end{figure}

If this is the case, it is now obvious that the firm's production strategy changes while it still faces the same internal demand $Y_t$ as before. In fact, the company is no longer merely exploiting favorable prices scenarios, but Equation \eqref{runcosts_pricemaker} (vs \eqref{runcosts_pricetaker}) introduces a direct feedback from the firm's own production to the price of electricity, which in turn discourages operation at minimum/maximum capacity unless strictly required by the internal demand. \\

Figure \ref{fig:simulation_x_1w_3d_pricemaker} clarifies this point: although the market price of electricity is temporarily low during day 4 on the left-hand side of Figure \ref{fig:price_Z}, we do not observe the same run to $P^{\text{min}}$ as we did with Figure \ref{fig:simulation_x_1w_3d}. 
This is because the firm suddenly cutting nuclear production would possibly result in increasing the price of electricity via
\begin{equation*}
    P_t \downarrow \quad \implies \quad M_t + Y_t - P_t = Z_t \uparrow  
\end{equation*}
thus ending up buying at a higher price. \\
Similarly, the company will not reach $P^{\text{max}}$ during day 5 on the right-hand side of Figure \ref{fig:simulation_x_1w_3d_pricemaker} despite the fact that the electricity price being high in Figure \ref{fig:price_Z} to avoid lowering their margins while selling their over-production.

\begin{figure}[H]
    \centering    
    \begin{subfigure}{0.49\textwidth}
        \centering
         \includegraphics[width=\linewidth]{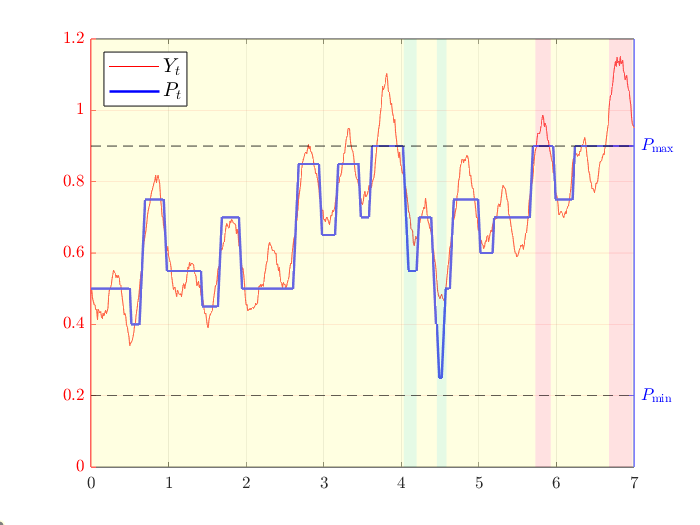}
        \label{subfig:PY1_pm}
    \end{subfigure}
    \hfill
    \begin{subfigure}{0.49\textwidth}
        \centering
        \includegraphics[width=\linewidth]{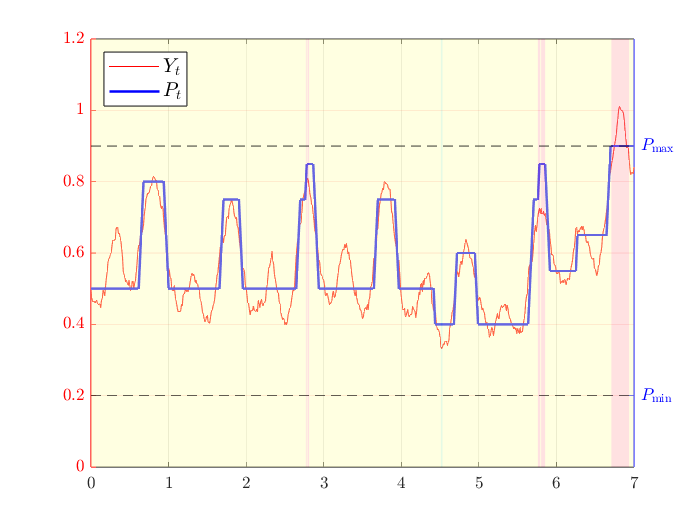}
        \label{subfig:PY2_pm}
    \end{subfigure}
    \caption{Two different realizations of the residual demand $\{Y_t\}_{t\in [0,T]}$ (red line) and the associated optimal production rate $\{P_t\}_{t\in [0,T]}$ (blue line). The red and green shaded regions indicate periods with high and low electricity prices, respectively,
while the yellow shaded region corresponds to the medium-price regime.}
    \label{fig:simulation_x_1w_3d_pricemaker}
\end{figure}

\subsection{Summary of simulated outcomes}

Table~\ref{tab:numerical-summary} summarizes the simulated outcomes associated with the optimal policies displayed in the left-hand panels of Figure \ref{fig:simulation_x_p_1w} and Figures~\ref{fig:price}-\ref{fig:simulation_x_1w_3d_pricemaker}. In the closed-economy case, the nuclear power plant tracks residual demand more closely than in the open-economy scenarios. The average absolute tracking error is \(0.03598\) in the closed economy, compared with \(0.08153\) in the price-taker case and \(0.05797\) in the price-maker case. This tighter tracking is accompanied by a substantially larger number of regime changes: the closed-economy simulation involves \(87\) switches, whereas the price-taker and price-maker simulations involve \(29\) and \(41\) switches, respectively. This reflects the fact that, in the absence of market access, physical adjustment of nuclear output is the main instrument available to manage residual-demand fluctuations. Market access changes the role of deviations between nuclear production and local residual demand. In the open-economy cases, shortages can be covered through purchases and excess production can be sold. This is reflected in the positive transaction volumes reported in Table~\ref{tab:numerical-summary}: purchases amount to \(0.4106\) in the price-taker case and \(0.2644\) in the price-maker case, while sales amount to \(0.1601\) and \(0.1414\), respectively. The possibility of trading therefore partly substitutes for nuclear maneuvering, as shown by the lower number of switches in the open-economy cases. The time spent in shortage and excess states is relatively similar across the three scenarios. Shortage occurs during \(57.14\%\), \(57.44\%\), and \(56.40\%\) of simulated time in the closed-economy, price-taker, and price-maker cases, respectively, while excess production occurs during roughly \(43\%\) of simulated time in all cases. The corresponding energy volumes, however, differ across market environments. Shortage energy is \(0.1812\) in the closed economy, \(0.4106\) in the price-taker case, and \(0.2644\) in the price-maker case. Excess energy is \(0.0707\), \(0.1601\), and \(0.1414\), respectively. Thus, although the time shares are similar, market access changes the magnitude and economic interpretation of imbalances. \\

The comparison between the price-taker and price-maker cases shows how price formation affects optimal operation. When the producer can affect market residual demand, the average tracking error falls from \(0.08153\) to \(0.05797\), purchases fall from \(0.4106\) to \(0.2644\), and the number of switches rises from \(29\) to \(41\). This is consistent with the fact that the price-maker internalizes the feedback from its own imbalance to the market price. The same nuclear technology is therefore operated differently depending on whether imbalances are absorbed domestically, traded at exogenous market prices, or allowed to affect the market price. \\

Finally, total costs are of comparable magnitude across the three simulations, ranging from \(1.200\) to \(1.266\) in model units. These values should not be interpreted as monetary estimates. They are mainly useful for comparing policies within the numerical environment considered here. The robust qualitative message is that shortage penalties, switching costs, ramping constraints, and market access jointly determine the extent to which nuclear production follows residual demand. \\

\begin{table}[H]
\centering
\caption{Summary statistics for simulated optimal load-following policies.}
\label{tab:numerical-summary}
\begin{tabular}{lrrr}
\toprule
Metric & Closed & Price taker & Price maker \\ 
\midrule
Total cost & 1.266 & 1.2 & 1.215 \\ 
Running cost & 1.242 & 1.191 & 1.203 \\ 
Switching cost & 0.02452 & 0.00898 & 0.01176 \\ 
$|P-Y|$ average & 0.03598 & 0.08153 & 0.05797 \\ 
Shortage energy & 0.1812 & 0.4106 & 0.2644 \\ 
Excess energy & 0.0707 & 0.1601 & 0.1414 \\ 
Shortage time & 57.14\% & 57.44\% & 56.40\% \\ 
Excess time & 42.71\% & 42.41\% & 43.45\% \\ 
Number of switches & 87 & 29 & 41 \\ 
At bounds & 18.45\% & 23.51\% & 20.54\% \\ 
Purchases energy & -- & 0.4106 & 0.2644 \\ 
Sales energy & -- & 0.1601 & 0.1414 \\ 
Purchase cost & -- & 0.1023 & 0.08008 \\ 
Sales revenue & -- & 0.03208 & 0.02825 \\ 
\bottomrule
\end{tabular}
\vspace{0.5em}
\parbox{0.94\textwidth}{\small Notes: All quantities are computed on the simulated paths. Cost variables are reported in
model units and include running and switching costs. Production and residual demand are expressed in normalized capacity units and the
cost coefficients are used for numerical comparison. \( |P-Y| \) is the time-average absolute tracking
error. Shortage and excess energy are time integrals of \((Y_t-P_t)^+\) and \((P_t-Y_t)^+\),
respectively. Shortage and excess time are percentages of simulated time. They sum to about 0.9985 as we start at $P_0 = Y_0$ and one time step (15 minutes) account for about 0.15\% of one week. 
At bounds is the combined percentage of time at \(P^{\min}\) or \(P^{\max}\).}
\end{table}

\section{Conclusion and policy implications}\label{sec:conclusion}

The increasing penetration of weather-dependent renewable generation changes the economic role of dispatchable technologies. As renewable output is dispatched with priority, the relevant uncertainty faced by conventional and low-carbon controllable producers is residual demand, namely electricity demand net of renewable production. This paper studies the operating problem of a producer that uses a load-following nuclear power plant to meet stochastic residual demand while accounting for technical limits, finite ramping capability, and costly changes in the direction of production. We formulate the problem as a finite-horizon optimal switching problem, characterize the associated value functions through a system of Hamilton--Jacobi--Bellman quasi-variational inequalities, and compute optimal policies by means of a monotone semi-Lagrangian scheme.

The numerical results highlight the economic trade-offs behind nuclear load following. In the closed-economy, the optimal policy is driven by the relative cost of shortage, excess production, switching, and maneuvering capability. When shortage costs are high, the plant tracks residual demand more closely and the producer is willing to switch regimes more frequently in order to reduce the probability of relying on fossil back-up or incurring reliability costs. Higher switching costs generate wider inaction regions and smoother production profiles, because frequent adjustments become less attractive. The ramping capability of the plant also matters, although its value depends on the way flexibility is represented: a higher ramping capability is valuable when residual demand changes rapidly, but its value is greatest when the producer can choose how much of that capability to use. If the ramping rate is fixed, a very large value of \(r\) may generate overshooting and additional regime changes rather than smoother tracking.

Allowing the producer to trade electricity in an external market changes the nature of the operating problem. In the open-economy case, nuclear production is no longer determined only by local residual demand. Market prices create an additional opportunity cost of producing, buying, or selling electricity. When external prices are low, it may be optimal to reduce nuclear production and cover part of local demand through purchases. When prices are high, maintaining or increasing nuclear output can be valuable even if domestic residual demand is low, because excess production can be sold. Market integration therefore affects the private value of load following and the social value of dispatchable low-carbon capacity. The same physical technology may be operated differently depending on balancing-market conditions, interconnection, and the price formation mechanism.

The policy implications are threefold. First, flexibility should be treated as a distinct economic attribute of low-carbon technologies. Energy-only revenues based on average production may fail to compensate the value of controllable capacity that reduces shortage risk, fossil back-up, or curtailment in periods of high residual-demand volatility. Market designs that remunerate availability, ramping capability, and balancing services can therefore be important when dispatchable low-carbon units are expected to support renewable integration.

Second, policy should account for the cost of providing flexibility. Load following is not costless: frequent changes in operating regimes may increase wear, maintenance requirements, fuel-management constraints, and operational complexity. A market design that encourages flexible nuclear operation should therefore balance the system value of responsiveness against the private and technical costs borne by plant operators. In the language of the model, switching costs are not a secondary detail; they shape the size of the inaction regions and the frequency with which the plant adjusts production.

Third, interconnection and market access can partly substitute for domestic flexibility, but they also modify incentives. In a closed system, excess low-carbon production has little value when it cannot be stored or exported, while shortage is costly because it requires fossil back-up or involuntary curtailment. In an open system, imports and exports make the opportunity cost of nuclear production state-dependent. This suggests that policies toward nuclear flexibility, storage investment, and interconnection should be evaluated jointly. The value of nuclear load following is likely to be higher in systems with limited hydro resources, limited storage, and weak interconnection, and lower in systems where reservoir hydropower, storage, or cross-border trade already provide abundant flexibility.

The framework developed in this paper is intentionally stylized and several extensions are natural. Future work could focus on a more realistic modelling of residual demand in continuous time and also try to estimate cost parameters on market data, introduce plant-level technical constraints in greater detail, model the ramping rate as a continuous control rather than a fixed operating regime, and allow the producer to affect market prices through strategic behavior. Another useful extension would be to combine nuclear load following with storage or multiple generating units, in order to quantify substitution and complementarity among different sources of low-carbon flexibility. These extensions would further connect the stochastic-control approach developed here with empirical questions on electricity-market design and decarbonization policy.

\subsection*{Funding Statement}
The authors acknowledge financial support under the National Recovery and Resilience Plan (NRRP):
\small{Mission 4, Component 2, Investment 1.1, Call for tender No. 1409 published on 14.9.2022 by the Italian Ministry of University and Research (MUR), funded by the European Union – NextGenerationEU – Project Title: 
Probabilistic Methods for Energy
Transition – CUP G53D23006840001 - Grant Assignment Decree No. 1379 adopted on 01/09/2023
by MUR. \\
F. Baschetti acknowledges financial support from the project PRICE (A New Paradigm for High Frequency Finance) financed by the Italian Ministry MUR under the program FIS 2021, Prot. FIS3055.

\subsection*{Acknowledgments}
We are grateful to Emilio Barucci, Peter Tankov, Tiziano Vargiolu for useful comments that helped us improve the paper. We also thank the participants of Energy Finance Italia 11 (University of Padova, 2026), the 2nd Latin American Congress on Industrial and Applied Mathematics (Universidad Técnica Federico Santa María, 2026),  the 3rd Workshop on Quantitative Methods for Green Finance (University of Palermo, 2026) and the special session \textit{Dynamic Models under Uncertainty in Economics and Finance} at the 15th American Institute of Mathematical Sciences Conference (University of Athens, 2026) for helpful feedback on a preliminary version of the work.

\subsection*{Declaration of competing interest}
The authors declare that they have no known competing financial interests or personal relationships 
that could have appeared to influence the work reported in this paper.

\subsection*{Declaration of generative AI and AI-assisted technologies in the manuscript preparation process}
During the preparation of this work, the authors used ChatGPT to assist with language editing and copy-editing. The authors reviewed and edited all AI-assisted text and take full responsibility for the content of the manuscript.

\bibliographystyle{alpha}
\bibliography{sample}

\appendix
\section*{Appendix}

\section{Proof of Theorem \ref{thm:dyn_prog_eqn}}\label{sec:proofs}

We only provide a sketch of the proof in the interest of brevity and simplicity (see Remark \ref{rmk:visco}). \\

Fix \(0<K<\infty\). For the regularized problem, the coefficients of the state process are continuous and Lipschitz in the state variable on the computational domain. Hence, for every admissible switching strategy, the controlled state equation admits a unique strong solution. Since the set of regimes \(\mathbb I\) is finite, the running and terminal costs are continuous, and the switching-cost matrix $C$ satisfies the no-free-loop condition \eqref{eq:noFreeLoop}, the finite-horizon switching problem is well posed and does not admit cost-free instantaneous switching cycles.

The dynamic programming principle, continuity of the value functions, and the viscosity characterization follow from standard results for finite-horizon Markovian optimal switching problems; see, for instance, \cite[Ch.~5, Sec.~5.3]{pham2009continuous}, \cite{DjehicheHamadenePopier2009}, and \cite{ElAsri2016}. We recall the formal argument, assuming smoothness of the value functions, to make the structure of the HJB-QVI explicit.

First of all, the boundary condition is satisfied as one has
$$
\lim_{t\to T} v^i(t,x) = v(T,x) = 0,\qquad \forall i\in \mathbb I.
$$
Take now $(t,x,i) \in [0,T) \times Q \times \mathbb{I}$ as a starting point. We consider two special cases on the set of possible actions available at $(t,x,i)$.
\begin{itemize}
    \item[(i)] Continue with regime $i$: let $h>0$ and, for any stopping time $\vartheta\in (t,T)$, denote $\vartheta_h:=\min(\vartheta, t+h)$. From the DPP, comparing $v^i$ with cost of not switching regime over $[t, t+h)$ one obtains 
    \begin{equation*}
        v^i(t,x) \leq \mathbb{E} \left[ \int_t^{\vartheta_h} f(X^{x,i}_s)  \dd s + v^{i}(\vartheta_h,X^{x,i}_{\vartheta_h}) \right].
    \end{equation*}
    from which by classical arguments involving the use of Ito's formula and mean value theorem one gets, for $h\to 0$, the ``no-switch'' inequality
  % $\{$ 'continue with regime $i$','switch to some $j \neq i$' $\}$.

  %   \item[(A1)] The 'no-switch' inequality \\

  %   Take $h>0$ and set $\theta = t+h$. \\ 
  %   The DPP gives
  %   \begin{equation*}
  %       v^i(t,x) \leq \mathbb{E} \left[ \int_t^{t+h} f(X^{x,i}_s) ds + \sum_{\tau_n \in [t,t+h)} c_{\iota_{n-1},\iota_n} + v^{I^i_{t+h}}(t+h,X^{x,i}_{t+h}) \right].
  %   \end{equation*}
  %   Hence, choosing not to switch over $[t,t+h)$, we have that
  %   \begin{equation*}
  %       v^i(t,x) \leq \mathbb{E} \left[ \int_t^{t+h} f(X^{x,i}_s) ds + v^i(t+h,X^{x,i}_{t+h}) \right].
  %   \end{equation*}
  %   Now
  %   \begin{equation*}
  %       \mathbb{E} \left[ \int_t^{t+h} f(X^{x,i}_s) ds \right] = h f(x) + o(h),
  %   \end{equation*}
  %   and Ito formula
  %   \begin{align*}
  %       v^i(t+h,X^{x,i}_{t+h}) & = v^i(t,x) + \int_t^{t+h} \left( \partial_t v^i + \mathcal{L}^i v^i \right)(s,X^{x,i}_s) ds + \int_t^{t+h} \nabla_x v^i(s,X^{x,i}_s)^\top \sigma^i(s,X^{x,i}_s) dW_s
  %   \end{align*}
  %   yields
    % \begin{align*}
    %     \mathbb{E}\left[ v^i(t+h,X^{x,i}_{t+h}) \right] & = v^i(t,x) + \mathbb{E} \left[ \int_t^{t+h} \left( \partial_t v^i + \mathcal{L}^i v^i \right)(s,X^{x,i}_s) ds \right] \\
    %     & = v^i(t,x) + h \left( \partial_t v^i + \mathcal{L}^i v^i \right)(t,x) + o(h).
    % \end{align*}
    % Putting everything together, we obtain
    % \begin{equation*}
    %     v^i(t,x) \leq v^i(t,x) + h \left[ \left( \partial_t v^i + \mathcal{L}^i v^i \right)(t,x) + f(x) \right] + o(h)
    % \end{equation*}
    % and finally taking the limit $h \to 0$ gives
    \begin{equation}\label{eqn:HJB_partial}
        -\left( \partial_t v^i + \mathcal{L}^iv^i \right) (t,x) - f(x) \leq 0.
    \end{equation}

    \item[(ii)] Immediate switch: let $j \in \mathbb{I}$, $j \neq i$, and consider the   control strategy that immediately (i.e. at time $t$) switches from regime $i$ to regime $j$, and then continues optimally from regime $j$, up to time $T$. By the very definition of $v^i$ one has
     $$
     v^i(t,x) \leq v^j(t,x) + c_{i,j}, \ \forall j \in \mathbb{I}.
     $$
     Therefore the following ``do-switch'' inequality holds
    \begin{equation}\label{eqn:HJB_obstacle}
        v^i(t,x) - \min_{ j \neq i} (v^j(t,x) + c_{i,j}) \leq 0.
    \end{equation}
\end{itemize}

Inequalities \eqref{eqn:HJB_partial} and \eqref{eqn:HJB_obstacle} prove the subsolution property
\begin{equation}\label{eq:subsol}
    \max \left\{ -\left( \partial_t v^i + \mathcal{L}^iv^i \right) (t,x) - f(x), v^i(t,x) - \min_{ j \neq i} (v^j(t,x) + c_{i,j}) \right\} \leq 0.
\end{equation}
We now wish to show that  equality actually holds.
% For this, we recognize that over a short time interval $[t,t+h]$ one can only choose between actions (A1): 'continue with regime $i$' and (B1): 'switch to some $j \neq i$'. Then the idea is to pick the action with the lowest cost, according to:
% \begin{equation}\label{eqn:DPP_small_h}
%     v^i(t,x) = \min \left\{ \mathbb{E} \left[ \int_t^{t+h} f(X^{x,i}_s) ds + v^i(t+h,X^{x,i}_{t+h}) \right], \min_{j \neq i} (v^j(t,x)+c_{i,j}) \right\}.
% \end{equation}
% \begin{itemize}
%     \item[(A2)] The 'no-switch' region \\  
Let us now assume that   $v^i(t,x) < \min_{j \neq i}(v^j(t,x)+c_{i,j})$, thus meaning that switching immediately at time $t$ is strictly worse than continuing with regime $i$. Therefore, we are in the so called ``no-switch region''. By continuity of the function $G(t,x):=v^i(t,x) - \min_{j \neq i}(v^j(t,x)+c_{i,j})$, we know that there exists $\rho > 0$ such that 
 \begin{equation*}
     G(s,\xi)<0  \qquad \forall (s,\xi) \in [t,t+\rho)\times B_\rho(x), 
\end{equation*}
where $B_\rho(x)=\{\xi \in Q : |x-\xi|<\rho\}$, meaning that in a sufficiently small time interval it is strictly optimal to continue with regime $i$. As a consequence, there exists a sufficiently small $h>0$ such that, by the DPP, 
 \begin{equation*}
 v^i(t,x) =  \mathbb{E} \left[ \int_t^{\vartheta_h} f(X^{x,i}_s) \dd s + v^i(\vartheta_h,X^{x,i}_{\vartheta_h}) \right],
 \end{equation*}
and the same reasoning as before (point (i)) implies
    \begin{equation*}
        -\left( \partial_t v^i + \mathcal{L}^iv^i \right) (t,x) - f(x) = 0,
    \end{equation*}
    leading to the equality in \eqref{eq:subsol}.
If instead
\[
v^i(t,x)=\min_{j\neq i}\left(v^j(t,x)+c_{i,j}\right),
\]
then the obstacle term in \eqref{eq:subsol} is equal to zero. Since the no-switch inequality \eqref{eqn:HJB_partial} still holds, the maximum in \eqref{eq:subsol} is again equal to zero. Therefore, in the smooth case, the value functions satisfy the system of variational inequalities. Uniqueness follows from the comparison principle for systems of variational inequalities with interconnected obstacles; see, for instance, \cite[Ch.~5, Sec.~5.3]{pham2009continuous}, \cite{DjehicheHamadenePopier2009}, and \cite{ElAsri2016}.
% So here
%     \begin{equation*}
%         \max \left\{ -\left( \partial_t v^i + \mathcal{L}^iv^i \right) (t,x) - f(x), v^i(t,x) - \min_j (v^j(t,x) + c_{i,j}) \right\} = 0.
%     \end{equation*}
    
%     \item[(B2)] The 'do-switch' region \\

%     Suppose $v^i(t,x) = \min_{j \neq i}(v^j(t,x)+c_{i,j})$, thus meaning that switching immediately at time $t$ is the best option. \\
%     At the same time, we know from before that
%     \begin{equation*}
%         -\left( \partial_t v^i + \mathcal{L}^iv^i \right) (t,x) - f(x) \leq 0.
%     \end{equation*}
%     Again
%     \begin{equation*}
%         \max \left\{ -\left( \partial_t v^i + \mathcal{L}^iv^i \right) (t,x) - f(x), v^i(t,x) - \min_j (v^j(t,x) + c_{i,j}) \right\} = 0.
%     \end{equation*}

% We conclude the proof by showing what happens at terminal time $t=T$, and immediately notice that $J^\alpha(T,x,i) = 0$. Then, $v^i(T,x) = 0$ by definition.

\begin{remark}\label{rmk:visco}
    Observe that we have been silently assuming that $v^i \in C^{1,2}([0,T] \times \mathbb{R}^d) \ i \in \mathbb{I}$ throughout the sketch above. The value functions $\{v^i, i\in \mathbb I\}$ are typically not this smooth in practice, even under strong regularity conditions on the data. This is due to the intrinsic nonlinear structure of the problem together with its degeneracy (there is indeed no diffusion in the first space direction). As already mentioned working in the framework of viscosity solutions allows to rigorously characterize the value function by means of the  HJB-QVI. We avoid these technicalities, which do not alter the spirit of the proof.  
\end{remark}

\section{Estimation}\label{sec:estimation}

\subsection{Theory}  
Estimation of an Ornstein-Uhlenbeck dynamics hinges on its discrete-time counterpart, the AR(1) process. \\

In fact, we have discrete observations
\begin{equation*}
    y_i = Y_{t_i}
\end{equation*}
over a time grid with mesh $\Delta$:
\begin{equation*}
    t_i = i \Delta, \qquad i = 1,\dots,N.
\end{equation*}
Let us therefore define $\phi := e^{-\kappa \Delta}$ so that
\begin{equation*}
\sigma^2_\Delta = \frac{\nu^2}{2\kappa} (1-\phi^2)
\end{equation*}
matches the true variance of $Y_t$. \\

Then, our discrete model reads as
\begin{equation}\label{eqn:AR1}
    x_i := y_i - \phi y_{i-1} = (1-\phi)\beta + \kappa \sum_j \left( \zeta_j I_j^{\zeta,i} + \eta_j I_j^{\eta,i} \right) + \sqrt{\sigma^2_\Delta} \varepsilon_i, \qquad \varepsilon_i \sim \mathcal{N}_{0,1} 
\end{equation}
where   
\begin{align*}
    & I_j^{\zeta,i} := \frac{[\kappa \cos(\omega_j t_i) + \omega_j \sin(\omega_j t_i)] - \phi [\kappa \cos(\omega_j t_{i-1}) + \omega_j \sin(\omega_j t_{i-1})]}{\kappa^2 + \omega_j^2} \\
    & I_j^{\eta,i} := \frac{[\kappa \sin(\omega_j t_i) - \omega_j \cos(\omega_j t_i)] - \phi [\kappa \sin(\omega_j t_{i-1}) - \omega_j \cos(\omega_j t_{i-1})]}{\kappa^2 + \omega_j^2}.
\end{align*}

Observe that, for any fixed $\phi$ (or $\kappa$), model \eqref{eqn:AR1} is amenable of treatment by OLS regression in the coefficients
\begin{equation*}
    \vartheta_\phi = \begin{bmatrix}
        \beta & \zeta_1 & \dots & \zeta_{|\omega|} & \eta_1 & \dots & \eta_{|\omega|}
    \end{bmatrix}^\top.
\end{equation*}
For this, we define 
\begin{equation*}
    B_\phi = \begin{bmatrix}
        b_1 & \dots & b_N
    \end{bmatrix}
    \in \R^{N+2|\omega|}
\end{equation*}
with
\begin{equation*}
    b_i = \begin{bmatrix}
        (1-\phi) & \kappa I_1^{\zeta,i} & \dots & \kappa I_{|\omega|}^{\zeta,i} & \kappa I_1^{\eta,i} & \dots & \kappa I_{|\omega|}^{\eta,i}
    \end{bmatrix}.
\end{equation*}
Hence, we conclude that
\begin{equation*}
    \hat{\vartheta}_\phi = (B_\phi^\top B_\phi)^{-1} B_\phi^\top x
\end{equation*}
with regression residuals
\begin{equation*}
    r_\phi = x - B_\phi \hat{\vartheta}_\phi.
\end{equation*}

The estimation procedure exploits the Gaussian transition structure implied by the Brownian-driven OU model, thus reducing maximum likelihood estimation to a simple least squares problem.
The procedure is as follows.
\begin{enumerate}
    \item Profile out $\vartheta$ via OLS regression as described above. \\
    In practice
    \begin{equation*}
        \phi^\star = \arg\min_\phi \sum (r_\phi)^2
    \end{equation*}
    identifies
    \begin{align*}
        & B := B_\phi^\star \\
        & \hat{\vartheta} := \hat{\vartheta}_{\phi^\star} = (B^\top B)^{-1} B^\top x \\
        & r := r_{\phi^\star} = x - B\hat{\vartheta}.
    \end{align*}
    \item Come back to 
    \begin{equation*}
        \hat{\kappa} = - \frac{\log \phi^\star}{\Delta}
    \end{equation*}
    and compute $$\hat{\sigma}^2_\Delta = \frac{1}{N} \sum r^2.$$ Then
    \begin{equation*}
        \hat{\nu} = \sqrt{\frac{2\hat{\kappa}\hat{\sigma}^2_\Delta}{1-(\phi^\star)^2}}.
    \end{equation*}
\end{enumerate} 

\subsection {Practice}
Figure \ref{fig:res_dem_est} shows that the estimated seasonal mean captures the main low-frequency variation in residual demand. 

\begin{figure}[H]
    \centering    
    \begin{subfigure}{0.49\textwidth}
        \centering
         \includegraphics[width=\linewidth]{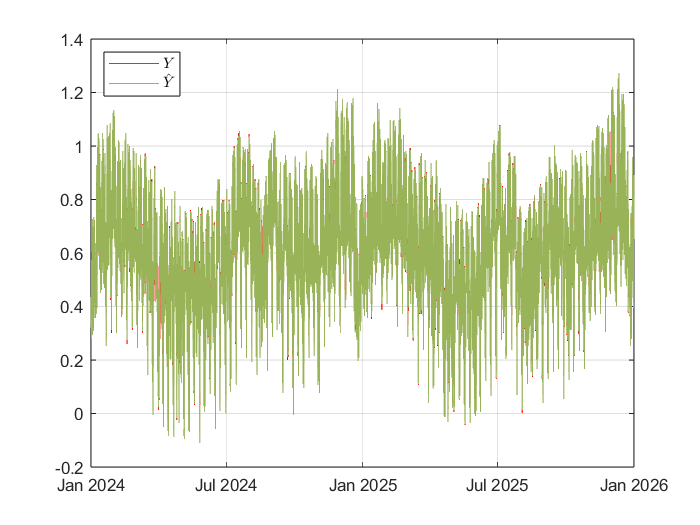}
         \caption{2024-2025}
        \label{subfig:large}
    \end{subfigure}
    \hfill
    \begin{subfigure}{0.49\textwidth}
        \centering
        \includegraphics[width=\linewidth]{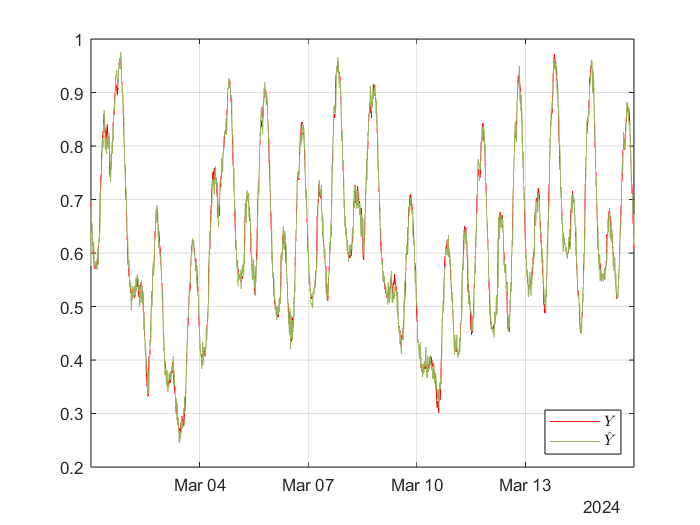}
        \caption{zoom}
        \label{subfig:zoom}
    \end{subfigure}
        \caption{Residual demand net of renewable resources in Italy. True (blue) vs Estimated (red)}
\label{fig:res_dem_est}
\end{figure}

However, the standardized residuals display heavier tails than implied by the Gaussian benchmark, as illustrated by the histogram and QQ-plot in Figure~\ref{fig:normal?}.

\begin{figure}[H]
    \centering    
    \begin{subfigure}{0.49\textwidth}
        \centering
         \includegraphics[width=\linewidth]{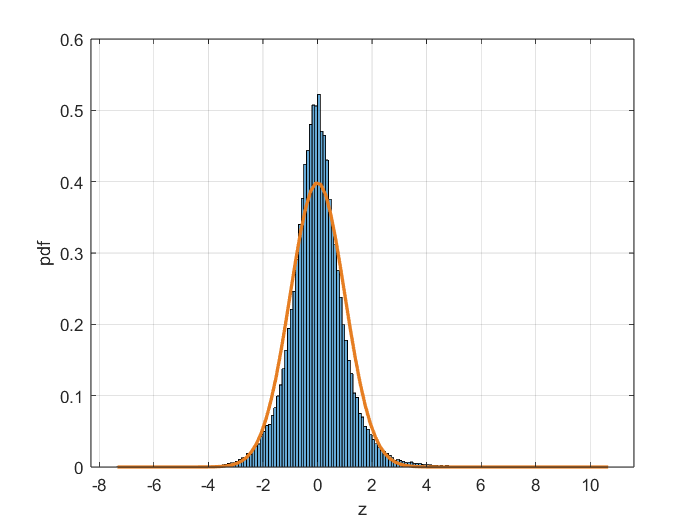}
         \caption{Histogram}
        \label{subfig:histo}
    \end{subfigure}
    \hfill
    \begin{subfigure}{0.49\textwidth}
        \centering
        \includegraphics[width=\linewidth]{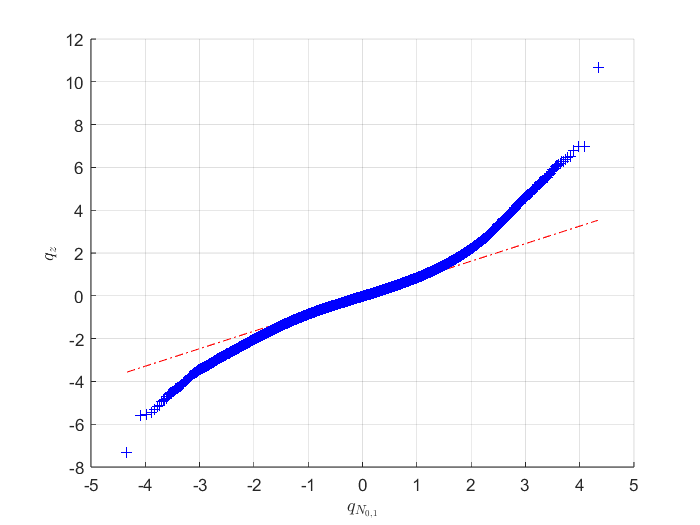}
        \caption{QQ-plot}
        \label{subfig:qqplot}
    \end{subfigure}
        \caption{Visual tests for Normality of the (standardized) regression residuals $z$}
\label{fig:normal?}
\end{figure}

Additionally, Figure \ref{fig:ACF} shows that standardized residuals and squared residuals are strongly serially correlated, thus suggesting that innovations exhibit a) a complex dependence structure and b) time-varying volatility.

\begin{figure}[H]
    \centering    
    \begin{subfigure}{0.49\textwidth}
        \centering
         \includegraphics[width=\linewidth]{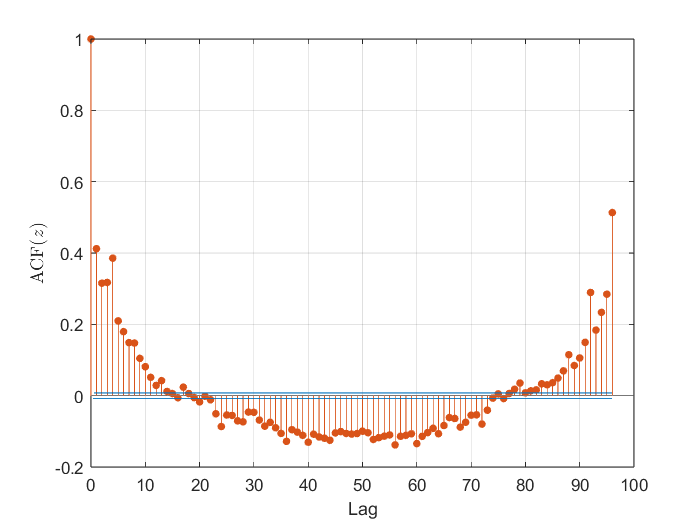}
         \caption{$z$}
        \label{subfig:ACF_z}
    \end{subfigure}
    \hfill
    \begin{subfigure}{0.49\textwidth}
        \centering
        \includegraphics[width=\linewidth]{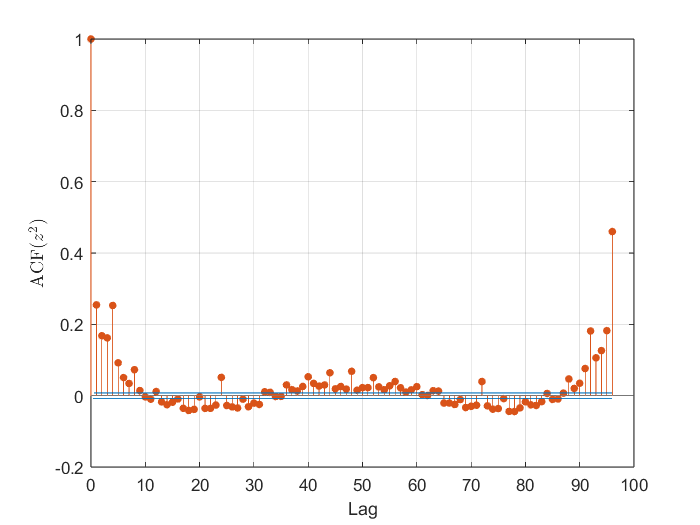}
        \caption{$z^2$}
        \label{subfig:ACF_z2}
    \end{subfigure}
        \caption{Autocorrelation function of the regression residuals (left) and squared residuals (right)}
\label{fig:ACF}
\end{figure}

In this sense, the fit in Figure \ref{fig:res_dem_est} should be interpreted as a parsimonious reduced-form approximation rather than as a full forecasting model. In fact, our numerical analysis simply wishes to study the operating implications of the switching problem under a tractable diffusion benchmark. \\

A richer statistical model of residual demand could include non-Gaussian innovations, time-varying volatility, or additional dependence structures. We leave such extensions to future work and use the OU specification as a transparent benchmark consistent with the diffusion framework underlying the HJB-QVI and the semi-Lagrangian scheme. The purpose of this work is in fact to define a simple -- yet reasonable -- framework for load following via a nuclear power plant. The qualitative structure of the switching problem does not depend on the specific parametric form chosen for residual demand. However, the numerical policy regions and simulated trajectories naturally depend on the calibrated dynamics.

\end{document}